\documentclass{amsart}
\usepackage{mathtools}
\usepackage{scalerel}
\usepackage[all]{xy}
\usepackage{amssymb}
\usepackage{amsmath}
\usepackage{cite}
\usepackage{fullpage}
\usepackage{mathrsfs}

\usepackage{url}
\usepackage{hyperref}
\numberwithin{equation}{subsection}
\numberwithin{figure}{subsection}
\usepackage{longtable}
\usepackage{tikz}
\usepackage{tikz-cd}
\usepackage{times}

\usepackage{setspace}

\newtheorem{theorem}{Theorem}[section]
\newtheorem{lemma}[theorem]{Lemma}
\newtheorem{proposition}[theorem]{Proposition}
\newtheorem{corollary}[theorem]{Corollary}
\theoremstyle{definition}
\newtheorem{definition}[theorem]{Definition}
\newtheorem*{remark*}{Remark}
\newtheorem{remark}[theorem]{Remark}
\newtheorem*{remarks*}{Remarks}
\newtheorem*{definition*}{Definition}
\newtheorem{example}{Example}[section]
\newtheorem{question*}{Question}[section]

\newcommand{\CC}{\mathbb{C}}

\newcommand{\ZZ}{\mathbb{Z}}

\newcommand{\PP}{\mathbb{P}}

\newcommand{\ga}{G^{\text{arith}}}
\newcommand{\gm}{G^{\text{geom}}}

\newcommand\restr[1]{\raisebox{-.5ex}{$|$}_{#1}}

\DeclareMathOperator{\Gal}{Gal}
\DeclareMathOperator{\sgn}{sgn}

\DeclareMathOperator{\Spec}{Spec}

\DeclareMathOperator{\im}{im}
\DeclareMathOperator{\id}{id}

\DeclareMathOperator{\Aut}{Aut}

\DeclareMathOperator{\Desc}{Desc}

\begin{document}
	
\title{Settled Elements in Arboreal Galois Groups of Quadratic PCF Polynomials}

\author{\"{O}zlem Ejder}
\address{Ko\c{c} University, College of Science}
\email{ozejder@ku.edu.tr}
\author{D\.ILBER Ko\c{c}ak }
\address{METU, Department of Mathematics}
\email{dkocak@metu.edu.tr}


\date{}

\begin{abstract}

Let $f(x) \in K(x)$ be a quadratic polynomial where $K$ is a field of characteristic not equal to $2$. The associated arboreal Galois representation of the absolute Galois group of $K$ acts on a regular rooted binary tree. Boston and Jones conjectured that, for $f \in \mathbb{Z}[x]$, the image of this representation contains a dense set of settled elements. Roughly speaking, a cycle of an automorphism $\tau$ of the tree is called \emph{stable} if its length strictly increases at each subsequent level, and $\tau$ is called \emph{settled} if the proportion of vertices contained in stable cycles goes to $1$ as the level goes to infinity.

In this article, we prove that, when $K$ is finitely generated over its prime field, the arithmetic iterated monodromy groups of postcritically finite quadratic polynomials in $K[x]$ with periodic postcritical orbits are densely settled. In the number field case, by a result of Benedetto–Ghioca–Juul–Tucker (2025), it follows that for infinitely many $a \in K$, the associated arboreal Galois representations are densely settled. In particular, our results apply to the arithmetic IMG of the Basilica map $f(x)=x^2-1$.

\end{abstract}

\date{}
\maketitle
\section{Introduction}

Let $K$ be a field of characteristic not equal to $2$ and let $\bar{K}$ denote a separable closure of $K$. Let $f(x)$ be a rational function of degree two defined over $K$, and fix a point $a \in K$ such that for every $n \geq 1$ the polynomial $f^n(x)-a$ has exactly $2^n$ distinct roots. 
The set of all iterated preimages of $a$ under $f$ can be organized into a rooted binary tree $T$: vertices at level $n$ correspond to solutions of $f^n(x)=a$, and edges connect each vertex to its image under $f$. 
Passing to the inverse limit over all levels yields an infinite rooted binary tree equipped with a natural action of the absolute Galois group $
G_K := \Gal(\bar{K}/K)$, arising from its action on the roots of the iterates of $f$. 

This action defines a continuous homomorphism
\[
\rho_{f,a} \colon G_K \longrightarrow \Omega := \Aut(T),
\]
called the \emph{arboreal Galois representation} associated with $(f,a)$. Its image, denoted $G_{f,a}$, is referred to as the \emph{arboreal Galois group}. These groups capture arithmetic information about the fields generated by iterated preimages and have been studied from a variety of perspectives in arithmetic dynamics; see, for example, \cite{Stoll92, Odoni2, Odoni85, Jonessurvey, JM14, S16}. 
Concrete descriptions and explicit computations appear in \cite{BEK, Arborealcubic, EjderBelyi, EKO, arithmeticbasilica, BHWL17, JKLMTW}.

One approach to studying arboreal Galois groups is to analyze how their elements act on the tree $T$. In particular, one can study how the cycle structure of an automorphism changes at higher levels of the tree. Following Boston and Jones \cite[Definition~3.2]{BJ-dedicata}, \cite[p.~216]{BJ-quarterly}, a cycle of an element $\tau \in \Omega$ at level $n$ is called \emph{stable} if its length increases at every higher level. An automorphism $\tau$ is called \emph{settled} if the proportion of vertices contained in stable cycles at level $n$ approaches $1$ as $n \to \infty$ (see Definition~\ref{defn:settled}). This notion was originally motivated by questions about polynomial factorization modulo primes \cite[Conjecture~2.2]{BJ-pams} and is important for studying the arithmetic structure of arboreal images.

When $K$ is a number field, the cycle structure of a Frobenius element at an unramified prime $p$ records the factorization type of the iterates of $f$ modulo $p$, and Boston and Jones \cite{BJ-pams} studied such elements from this point of view, observing that they are expected to be settled. Since Frobenius elements are dense in $G_K$ by the Chebotarev density theorem, they conjectured that $\im(\rho_{f,a})$ contains a dense set of settled elements. We refer to \cite{BJ-pams, BJ-dedicata, BJ-quarterly} for more details on the motivation behind these definitions.


A subgroup of $\Omega$ is called \emph{densely settled} if it contains a dense set of settled elements. 
Boston and Jones \cite[pg.~30]{BJ-dedicata} conjectured that the arboreal Galois groups of quadratic polynomials are densely settled.

Settledness has been studied in the case of polynomials---particularly by Boston--Jones \cite{BJ-pams, BJ-dedicata, BJ-quarterly} and Cortez--Lukina \cite{CortezLukina}. In contrast, in the non-polynomial setting, the first author \cite{Ejder2025} showed that there exist many quadratic rational functions $f$ for which the groups $G_{f,a}$ do not contain \emph{any} settled elements.

We define the postcritical orbit $P$ of a polynomial $f$ by
\[
P := \{ f^n(c) \mid c \text{ is a critical point of } f,\ n \geq 1 \}.
\]
If the set $P$ is finite, then $f$ is called \emph{postcritically finite} (PCF). 

The group $G_{f,a}$ is contained in a self-similar subgroup called the \emph{arithmetic iterated monodromy group} (IMG) associated to $f$. 
Roughly speaking, this group can be described as a quotient of the étale fundamental group of $\PP_{K}^1 \setminus P$, and it captures both arithmetic and geometric information about the tower of covers of $\PP^1$ obtained by iterating $f$. 
The arithmetic IMG contains a subgroup known as the \emph{geometric} IMG; when $K$ is contained in $\CC$, it can be calculated as a quotient of the profinite completion of the topological fundamental group attached to $f$ over $\CC$. Precise definitions are given in Section~\ref{mon}. 


The article \cite{CortezLukina} studies this problem in a wider
perspective. They consider settled elements in the automorphism group of a
$d$-ary tree, see Definitions~1.1--1.3, so their framework applies to
polynomials of any degree $d \geq 2$. In particular, they show (Remark 6.11)
that some of the normalized Dynamical Belyi maps studied in \cite{BEK} have
densely settled IMG's. When $d=2$, they studied the settled elements in the (profinite) iterated monodromy groups 
(IMG) of quadratic PCF polynomials with strictly preperiodic postcritical orbit $P$. They established dense settledness in one specific configuration of the postcritical orbit and provided substantial evidence for the general strictly preperiodic case, and raised the problem of understanding the periodic case (see \cite[Problem~1.17]{CortezLukina}).

In this article, we resolve this problem completely by proving that, when $K$ is finitely generated over its prime field and $f(x) \in K[x]$ is a quadratic polynomial with a periodic postcritical orbit, its arithmetic IMG is densely settled. 


\begin{theorem}\label{intro:thm1}
Let $K$ be a field of characteristic not equal to $2$. Assume that $K$ is finitely generated over its prime field. Let $f(x) \in K[x]$ be a PCF quadratic polynomial with a periodic postcritical orbit. Then the arithmetic iterated monodromy group of $f(x)$ is densely settled.
\end{theorem}

The arithmetic IMGs of such maps were studied in detail by Pink \cite{Pinkpolyn}. A prominent example is the Basilica map $f(x) = x^2 - 1$, whose arithmetic IMG was explicitly described in \cite{arithmeticbasilica}.

We obtain the following corollary by \cite[Theorem~1.2]{BGJT2025s}. See also \cite[Theorem~1.5]{BGJT2025p} for a similar result.

\begin{corollary}
Let $K$ be a number field and let $f(x) \in K[x]$ be a PCF quadratic polynomial with a periodic postcritical orbit. For infinitely many $a \in K$, the associated arboreal Galois representation of $(f,a)$ is densely settled.
\end{corollary}

Moreover, since a finite-index subgroup of a densely settled group is itself densely settled, we note that if \cite[Conjecture~2]{arithmeticbasilica} holds, then the arboreal Galois groups associated to $(x^2-1,a)$ 
 are densely settled for all but finitely many $a \in K$.

The geometric IMG of a PCF quadratic polynomial with a periodic postcritical orbit is determined up to conjugacy by Pink. See also the recent work of Adams and Hyde \cite{AdamsHyde2025}. Our approach is two-fold. First, we exploit the fact that the arithmetic IMG contains the normalizer of any odometer lying in the geometric IMG. We therefore first give an explicit description of settled elements inside the normalizer of an odometer.

More precisely, let $\gamma \in \Omega$ be defined recursively by $\gamma = (\gamma,\id)\sigma$. The normalizer of the topological group generated by $\gamma$ is isomorphic to $\ZZ_2 \rtimes \ZZ_2^{\times}$, and for each $k \in \ZZ_2^{\times}$, there is $z \in \Omega$ such that $z\gamma z^{-1}=\gamma^k$. 
Our first main result is the following (Corollary~\ref{zksummary}).

\begin{theorem}\label{intro:thm2}
Let $k \neq \pm 1 \in \ZZ_2^{\times}$ and let $z \in \Omega$ be such that $z\gamma z^{-1}=\gamma^k$. Let $\nu$ be the $2$-adic valuation of $(k^2-1)/4$. Then for any $n \geq \nu +1$, the number of vertices at level $n$ that do not lie in a stable cycle of $z$ is at most $2^{\nu +1}$.
\end{theorem}

While Cortez and Lukina \cite{CortezLukina} showed that the normalizer of an odometer is densely settled in a more general $d$-ary setting, our result provides a finer, element-by-element description of settledness in the binary case. Although related estimates appear implicitly in their work (for example, in the proof of Lemma~7.6), the statement above does not appear explicitly there. Moreover, our proof relies on the recursive structure of $\gamma$ rather than on topological dynamics.

Next, we show that the odometer $\gamma$ occurs naturally in the chosen model for the geometric IMG. Using Pink’s description of arithmetic IMGs, we prove that the normalizer of the group generated by $\gamma$ is contained in the arithmetic IMG.

Finally, we introduce the notion of \emph{descendants} of an automorphism in $\Omega$, which record the induced actions on subtrees under the standard wreath recursion. When an element acts nontrivially on the first level, passing to its descendants allows one to replace it by a simpler representative in its conjugacy class. A central innovation of this paper is the use of descendants to define stable blocks, which serve as the main tool for constructing settled elements. See Definition~\ref{defn:descendant} and Definition~\ref{defn:stableblock} for precise formulations. Stable blocks play an important role in the proof of Theorem~\ref{intro:thm1}.
We prove the following key result. 

\begin{theorem}\label{intro:thm3}
Let $K$ be a field of characteristic not equal to $2$, and let $f(x)\in K[x]$ be a PCF quadratic polynomial with a periodic postcritical orbit.
Then any stable block in the arithmetic IMG of $f(x)$ is settled.

\end{theorem}

The proof of Theorem~\ref{intro:thm1} is completed by an induction argument on the length of an associated word. Theorem~\ref{intro:thm2} provides the base step of induction, whereas Theorem~\ref{intro:thm3} is used in the induction step.

\section*{Acknowledgements}
The authors were supported by the TUBITAK project 124F203. The authors thank Olga Lukina for helpful comments on the manuscript. The authors also thank the referee for a careful reading of the manuscript and for many helpful comments and suggestions that improved the paper.

\section{Notation and background}\label{background}

\subsection{Automorphism group of a regular rooted binary tree}\ \\
Let $X=\{0,1\}$, and consider the infinite rooted binary tree $T$, whose vertices are formed by finite words over $X$. The root is represented by the empty word, and any vertex $v$ is adjacent to each $vx$ for $x \in X$. For $n \geq 1$, we denote by $T_n$ the finite rooted subtree consisting of all words of length at most $n$. The collection of words of length $n$ constitutes the \emph{level $n$} of $T$, and we write $V_n$ for the set of vertices at this level. Let $V$ denote the set of all vertices of the tree. 

An isomorphism between two rooted binary trees $T$ and $T'$ is a bijection of vertices that respects adjacency. When $T=T'$, the full group of such isomorphisms is called the automorphism group of $T$, denoted by $\Omega$. Likewise, $\Omega_n$ stands for the automorphism group of the finite tree $T_n$. For each $n \geq 1$, there is a canonical restriction map
\begin{equation}\label{not:pi}
\pi_{n}\colon \Omega\to \Omega_n,
\end{equation}
which restricts an automorphism of $T$ to its action on $T_n$ (levels $0,1,\ldots,n$). More generally, for $m \geq n$, the natural projection $\Omega_m \to \Omega_n$ is denoted by $\pi_{m,n}$. For simplicity, we use $\pi_n$ whenever the domain is unambiguous.

Given $\gamma \in \Omega$, we denote its restriction to $T_n$ by $\gamma \restr{T_n}$. Two automorphisms $\gamma,\gamma' \in \Omega$ coincide precisely when $\gamma \restr{T_n} = \gamma' \restr{T_n}$ for every $n \geq 1$. For any subgroup $H \leq \Omega$, we define $H_n := \pi_n(H) \subseteq \Omega_n$.

The automorphism group $\Omega$ is the inverse limit of the system $((\Omega_n)_{n \geq 1}, (\pi_{m,n})_{m\geq n})$. It follows that $\Omega$ is a complete profinite group. For each $n \geq 1$, we define 
\[ U_n := \ker(\pi_n: \Omega \to \Omega_n).\] 

The family $\{U_n\}_{n \geq 1}$ forms a neighborhood basis of the identity in $\Omega$ under the profinite topology on $\Omega$.
We also note that since each $\Omega_n$ is a $2$-group, $\Omega$ itself is a pro-$2$ group.  

\subsection{The $2$-adic integers}\ \\
We write $\ZZ_2$ for the ring of $2$-adic integers, and $\ZZ_2^\times$ for its multiplicative units. Let $v_2$ denote the $2$-adic valuation on $\ZZ_2$. An element $a \in \ZZ_2$ is written uniquely in the form $\sum_{i=0}^{\infty}a_i2^i$, where $a_i \in \{0,1\}$. The natural embedding defined by the $2$-adic expansion of integers identifies $\ZZ$ as a subring of $\ZZ_2$. 

The $2$-adic valuation of a non-zero element $a$ is the smallest $i$ with $a_i\neq 0$. We assume that $v_2(0)=\infty$. The ring $\ZZ_2$ has a unique maximal ideal which is given by $$\{ a \in \ZZ_2 \; | \; v_2(a)\geq 1\}.$$ Hence $\ZZ_2^{\times}$ is the set of elements $a \in \ZZ_2$ with $v_2(a)=0$. Let $a \in \ZZ_2^{\times}$. Then exactly one of $v_2(a \pm 1)$ is equal to $1$ while the other is $\geq 2$. In other words, for $a \in \ZZ_2^{\times}$, the element $\frac{a \pm 1}{2}$ is an element of $\ZZ_2$. In particular, we have $v_2(a^2-1)> 2$.

We also note that if $a,b \in \ZZ_2$ such that $v_2(a)=v_2(b)=v$, then $v_2(a+b)>v$. This is not true for $p$-adic integers if $p>2$.

\subsection{Closed subgroups of $\Omega$} \ \\
Let $S$ be a subset of $\Omega$. An element $\alpha \in \Omega$ is in the closure of $S$ if and only if for every $n\geq 1$, there is some $g_n \in S$ such that $\alpha^{-1}g_n \in U_n$, i.e.,  $$\alpha \restr{T_n}=g_n \restr{T_n}.$$

The closure of the cyclic subgroup generated by an element $\alpha \in \Omega$ is 
\[ \langle\langle \alpha \rangle\rangle := \{ \alpha^k \mid k \in \ZZ_2 \}.\]

Let $k=\sum_{i=0}^{\infty}a_i2^i \in \ZZ_2$ and let $\alpha \in \Omega$. For any $n\geq 0$,  let $k_n=\sum_{i=0}^{n-1}a_i2^i$, then the automorphisms $\alpha^k$ and $\alpha^{k_n}$ agree on level $n$ of the tree.

For $c_1,\dots,c_k \in \Omega$, we denote by $\langle\langle c_1,\dots,c_k \rangle\rangle$ the closure of $\langle c_1,\dots,c_k \rangle$ inside $\Omega$. A closed subgroup $H$ is said to be topologically generated by $c_1,\dots,c_k$ when
\[
H = \langle\langle c_1,\dots,c_k \rangle\rangle.
\]
If $H \leq \Omega$ is closed, then $H$ arises as the inverse limit of $(H_n)_{n\geq1}$. Consequently, every closed subgroup of $\Omega$ is also a pro-$2$ group.  

\subsection{The semi-direct product structure of $\Omega$}\ \\

Let $T_0$ and $T_1$ be the subtrees rooted at the two vertices of level one. Embedding the product $\Omega(T_0)\times \Omega(T_1)$ into $\Omega$ is achieved by letting $(\gamma_0,\gamma_1)$ act trivially on the first level and as $\gamma_i$ on $T_i$ for $i=0,1$.  

Identifying $T$ with $T_0$ via $w \mapsto 0w$, and with $T_1$ via $w\mapsto 1w$, gives rise to group isomorphisms, yielding an embedding
\[
\Omega \times \Omega \hookrightarrow \Omega.
\]
Its image consists precisely of automorphisms that act trivially on the first level.  

Thus, $\Omega$ decomposes as the semi-direct product
\[
\Omega \simeq (\Omega \times \Omega) \rtimes \langle \sigma \rangle,
\]
where $\sigma$ swaps the two subtrees $T_0$ and $T_1$. Accordingly, any element $\gamma \in \Omega$ admits a representation $\gamma=(\gamma_0,\gamma_1)\tau$, with $\gamma_0,\gamma_1 \in \Omega$ and $\tau \in S_2$.  

The group law is given by
\begin{equation}\label{semi}
(\gamma_0,\gamma_1)\tau \cdot (\gamma_0',\gamma_1')\tau' = (\gamma_0\gamma'_{(0)\tau},\, \gamma_1\gamma'_{(1)\tau})\,\tau\tau',
\end{equation}
and inversion is expressed as
\[
((\gamma_0,\gamma_1)\tau)^{-1} = (\gamma^{-1}_{(0)\tau},\, \gamma^{-1}_{(1)\tau})\,\tau^{-1}.
\]

For $n \geq 2$, one similarly has
\[
\Omega_n \simeq (\Omega_{n-1} \times \Omega_{n-1}) \rtimes S_2,
\]
where $S_2 \simeq \langle \sigma \rangle$.  

If $\gamma=(\gamma_0,\gamma_1)\tau$, then for $x\in X$ and any word $v$,
\[
(xv)\gamma = (x)\tau \cdot (v)\gamma_x.
\]
Here permutations in $S_2$ act from the right. In particular, $\sigma=(\id,\id)\sigma \in \Omega$, where $\id$ denotes the identity.  

\begin{example}
Let $\gamma=(\gamma,\id)\sigma$. Then 
$(10)\gamma=00$, $(00)\gamma=11$, $(01)\gamma=10$, and $(11)\gamma=01$.
\end{example}

\begin{proposition}\label{prop:conj}
Let $\alpha=(\alpha_0,\alpha_1)\tau$ be an element of $\Omega$.
\begin{enumerate}
    \item Assume $\tau=\id$. Then $\alpha$ is conjugate to $(\beta_0,\beta_1)$ if and only if either $\alpha_0,\alpha_1$ are conjugate to $\beta_0,\beta_1$ or $\beta_1,\beta_0$ respectively. 
    \item Assume $\tau\neq \id$. Then $\alpha$ is conjugate to $(\alpha_0\alpha_1,\id)\sigma$ in $\Omega$.
\end{enumerate}
\end{proposition}
\begin{proof}
This is a straight forward calculation. See \cite{Pinkpolyn} or \cite{Ejder2025} for a proof.
\end{proof}

\subsection{Sign of an element in $\Omega$}\ \\
For every $n \geq 1$, $\Omega_n$ acts faithfully on the $n$-th level of $T$, giving an embedding into $S_{2^n}$. Denote by $\sgn_n$ the sign homomorphism of the induced permutation:
\begin{equation}
\sgn_n\colon \Omega \to \{\pm1\}.
\end{equation}

For example, the element $\sigma\restr{T_n}$ decomposes into $2^{n-1}$ disjoint transpositions, hence
\begin{equation}\label{sgn:sigma}
\sgn_n(\sigma)=
\begin{cases}
-1 & n=1,\\
1 & n>1.
\end{cases}
\end{equation}

\begin{remark}\label{sgndecomp}
    If $\alpha=(\alpha_0,\alpha_1)\tau \in \Omega$, then $\sgn_n(\alpha)=\sgn_{n-1}(\alpha_0)\sgn_{n-1}(\alpha_1)=\sgn_{n-1}(\alpha_0\alpha_1)$ for all $n\geq 2$.
\end{remark}
 
 \subsection{Odometers}\ \\
An automorphism $\alpha \in \Omega$ is called an \emph{odometer} if $\alpha$ acts as a $2^n$ cycle on every level $n$ of the tree. Equivalently, $\alpha$ is an odometer if $\alpha \restr{T_n}$ has order $2^n$ for all $n\geq 1$. 

We will define the automorphism $\gamma \in \Omega$ recursively as follows. Set
\[
\gamma\restr{T_1}=\sigma,
\]
and for $n\ge2$ define
\[
\gamma\restr{T_n}
=
(\gamma\restr{T_{n-1}},\id)\sigma.
\]

Taking the limit as $n \to \infty$, we obtain that $\gamma=(\gamma,\id)\sigma$. The recursively defined element $\gamma=(\gamma,\id)\sigma$ is an example of an odometer and it is called the standard odometer.

\begin{proposition}\label{prop:odosgn}\cite[Proposition~3.1]{Ejder2025}
    An element $\alpha \in \Omega$ is an odometer if and only if $\sgn_n(\alpha)=-1$ for all $n\geq 1$. 
\end{proposition}

This does not hold for $d$-ary trees where $d>2$. We also note that any conjugate of an odometer in $\Omega$ is again an odometer since the cycle structure is preserved under conjugacy.

\subsection{Monodromy groups}\label{mon}\ \\

Let $K$ be a number field and $\bar{K}$ be a separable closure of $K$.  Let $f\colon\PP^1_K \to \PP^1_K$ be a morphism of degree $d$ defined over $K$. We denote the set of critical points of $f$ by $C$ and the forward orbit of the points in $C$ by $P$, i.e.
\[ P:=\{ f^n(c) \mid n\geq 1, c \in C\}. \]
We call $P$ the \emph{postcritical set} of $f$. 
A rational function is called postcritically finite (PCF) if the postcritical set $P$ of $f$ is finite.
For $n\geq 1$, the iterate $f^n$ is a connected unramified covering of $\PP^1_K \backslash P$, hence it is determined by the monodromy action of $\pi_1^{\acute{e}t}(\PP_K^1 \backslash P, x_0)$ on $f^{-n}(x_0)$ up to isomorphism, where $x_0 \in \PP^1_K \backslash P$. Let $T_{x_0}$ be the tree defined as follows: it is rooted at $x_0$, the leaves of $T_{x_0}$ are the points of $f^{-n}(x_0)$ for all $n\geq 1$, and the two leaves $p,q$ are connected if $f(p)=q$. By taking the inverse limit over $n$, associated monodromy defines a representation 
\begin{equation}\label{eq:rho}
 \rho_{f,x_0}\colon \pi_1^{\acute{e}t}(\PP_K^1 \backslash P, x_0) \to \Omega(T_{x_0})
\end{equation} 
where $\Omega(T_{x_0})$ denotes the automorphism group of $T_{x_0}$. We note that if $y_0$ is another point in $\PP^1_K\backslash P$, there is an isomorphism between the trees $T_{x_0}$ and $T_{y_0}$. Furthermore, $T_{x_0}$ and $T_{y_0}$ are isomorphic to the tree $T$ we defined earlier. Under these isomorphisms, $\im(\rho_{f, x_0})$ and $\im(\rho_{f, y_0})$ are conjugate subgroups of $\Omega$. Hence, we have:
\begin{equation}\label{eq:rhoOmega}
 \rho_{f}\colon \pi_1^{\acute{e}t}(\PP_K^1 \backslash P, x_0) \to \Omega
\end{equation} 

We call the image of $\rho_f$ the \emph{arithmetic iterated monodromy group} of $f$ and denote it by $G^{\text{arith}}(f)$. One can also study this representation over $\bar{K}$ and obtain 
\[ \pi_1^{\acute{e}t}(\PP_{\bar{K}}^1 \backslash P, x_0) \to \Omega(T).
\]
 We call the image of the map in this case the \emph{geometric iterated  monodromy group} and denote it by $\gm(f)$. The groups $\ga(f)$ and $\gm(f)$ fit into an exact sequence as follows:
 \begin{equation}\label{eq:exact}
   \begin{tikzcd}
1\arrow{r}  &  \pi_1^{\acute{e}t}(\PP_{\bar{K}}^1 \backslash P, x_0) \arrow{r} \arrow{d}& \pi_1^{\acute{e}t}(\PP_K^1 \backslash P, x_0) \arrow{r} \arrow{d} &\Gal(\bar{K}/K)\arrow{r} \arrow{d} &1  \\
   1 \arrow{r} &   G^{\text{geom}}(f) \arrow{r} &G^{\text{arith}}(f) \arrow{r}  &\Gal(F /K) \arrow{r} &1 \\
     \end{tikzcd}
  \end{equation}
for some field extension $F$ of $K$. We call this field $F$ the field of constants for $f$. 

The groups $G^{\text{geom}} (f)$ and $G^{\text{arith}}(f)$ are profinite self-similar subgroups of $\Omega$. Also note that $G^{\text{geom}}(f)$ is a normal subgroup of $G^{\text{arith}}(f)$. 

\subsection{Arboreal Galois groups} \ \\
Let $K$ be a number field, and let $\bar{K}$ be a separable closure of $K$. Let $f \colon \PP^1_K \to \PP^1_K$ be a morphism of degree $d$ defined over $K$, and fix $\alpha \in K$. Consider the morphism $s: \Spec(K) \to \PP_K^1$ given by $\alpha$. Let $\bar{\alpha}$ denote the image of $s$ and assume that it does not belong to $P$.
The functoriality of the \'etale fundamental group induces the homomorphism $G_K \to \pi_1^{\acute{e}t}(\PP_{K}^1\backslash P, \bar{\alpha})$. Composing this with $\rho_f$, we obtain
\[
\rho_{f,\alpha} \colon G_K \to \Omega(T_{\bar{\alpha}}).
\]

The image of $\rho_{f,\alpha}$ is called the \emph{arboreal Galois group} attached to the pair $(f, \alpha)$. We will denote this group by $G_{f,\alpha}$.
Since changing the base point is equivalent to conjugation by an element of $\Omega$, arboreal Galois group $G_{f,\alpha}$ is conjugate to a subgroup of the arithmetic IMG of $f$.

We may also describe $\rho_{f,\alpha}$ as follows:  the set of preimages of $\alpha$ under the iterates of $f$ forms a regular rooted binary tree $T_\alpha$, and the absolute Galois group $G_K = \Gal(\bar{K}/K)$ acts naturally on this tree. The tree $T_\alpha$ is isomorphic to the standard binary tree $T$, and under this isomorphism we obtain the representation $\rho_{f,\alpha}: G_K \to \Omega$.

\section{Settled elements}

In this section, we give a precise definition of settledness and prove several basic properties of settled elements, including invariance under conjugation and under taking squares. We also introduce one of the main tools used in this article: the notion of descendants of an automorphism in $\Omega$. This concept was first introduced by Pink \cite{Pinkorbitlength} in a different setting. Descendants will play a central role in the arguments of the final two sections.

\begin{definition}\label{defn:settled}
   Let $\alpha$ be an automorphism of the tree $T$. Let $n \geq 0$, and let $v$ be a vertex of $T$ at level $n$. We say $v$ is in a \emph{stable cycle of length} $k \geq 1$ of $\alpha$ if:
\begin{enumerate}
\item the orbit $\{(v)\alpha^i \mid i\geq 0\}$ has $k$ distinct elements.
\item for every $m > n$, all vertices in $T_m$ lying above $\{(v)\alpha^j \mid j \geq 0\}$ are in the same cycle of length $2^{m-n}k$.
\end{enumerate}

An automorphism $\alpha \in \Omega$ is called \textit{settled} if 
 $$ \lim_{n\rightarrow \infty} \frac{|\{v \in V_n \; :\; v \; \text{is in a stable cycle of $\alpha$}\}|}{|V_n|}=1 .$$
   An automorphism $\alpha \in \Omega$ is called \textit{strongly settled} if there exists $n\geq 1$ such that every vertex in $V_n$ is in a stable cycle of $\alpha$. A subgroup $H$ of $\Omega$ is called \emph{densely settled} if the set of settled elements is dense in $H$ under the profinite topology.

\end{definition}

 Let $\alpha \in \Omega$. Denote by $s_n(\alpha)$ the number of vertices at level $n$ that are in a stable cycle of $\alpha$. By definition, if $v \in V_n$ is in a stable cycle, then any vertex in $V_{n+1}$ lying above one of
$(v)\alpha, \dots, (v)\alpha^{k-1}$ is also in a stable cycle. Hence, for any $n\geq 1$,
\[
s_{n+1}(\alpha) \ge 2 s_n(\alpha).
\]
This implies that the sequence $(s_n(\alpha)/2^n)_{n\geq 1}$ is non-decreasing and bounded above by $1$, and therefore it converges.

\begin{example}
Let $\alpha$ be an odometer in $\Omega$ and let $k\in \ZZ_2$. If $v_2(k)=0$, then
$$\sgn_n(\alpha^k)=\sgn_n(\alpha)^{k}=-1$$
for all $n\geq 1$. Hence $\alpha^{k}$ is also an odometer. On the other hand, if $v_2(k)=m\geq 1$, then $\alpha^k=\alpha^{2^mk'}$ for some $k' \in \ZZ_2^{\times}$. Using the first observation, we may assume that $k=2^m$ for some $m\geq 1$. Then $\alpha^k\restr{T_n}$ is the product of $2^m$ cycles of length $2^{n-m}$ for all $n\geq 1$. 

Hence for any $k \in \ZZ_2$, $\alpha^k$ is strongly settled.
\end{example}

We first show that settledness is preserved under conjugacy.

\begin{proposition}\label{prop:conjsettled}
   Let $\alpha, \alpha' \in \Omega$ be conjugate to each other. Then $ s_n(\alpha)=s_n(\alpha')$ for all $n\geq 1$. Hence
   $\alpha$ is (strongly) settled if and only if $\alpha'$ is (strongly) settled.
\end{proposition}
A more general version of Proposition~\ref{prop:conjsettled} is given in \cite[Proposition 6.14]{CortezLukina}.
\begin{proof} 
Let $\alpha'=g \alpha g^{-1}$ for some $g \in \Omega$. Since conjugation preserves the cycle structure, a vertex $v \in V_n$ is in a stable cycle of $\alpha$ if and only if $g(v)$ is in a stable cycle of $\alpha'$. 
\end{proof}

\begin{proposition}\label{prop:settledsections}
Let $\alpha = (\alpha_0, \alpha_1) \tau \in \Omega$.
\begin{enumerate}
    \item If $\tau = \id$, then 
    \[
    s_n(\alpha) = s_{n-1}(\alpha_0) + s_{n-1}(\alpha_1) \quad \text{for all } n\ge 1.
    \]
    \item If $\tau = \sigma$, then 
    \[
    s_n(\alpha) = 2 s_{n-1}(\alpha_0 \alpha_1) \quad \text{for all } n\ge 1.
    \]
\end{enumerate}
\end{proposition}
\begin{proof}
If $\tau = \id$, then $\alpha$ acts independently on the two subtrees at level $1$, giving $s_n(\alpha) = s_{n-1}(\alpha_0) + s_{n-1}(\alpha_1)$.

If $\tau = \sigma$, then $\alpha$ is conjugate to $(\alpha_0 \alpha_1, \id) \sigma$ by Proposition~\ref{prop:conj}. 
Using Proposition~\ref{prop:conjsettled}, we may assume $\alpha = (\alpha_0\alpha_1, \id) \sigma$. Then for any word $v$ of length $n-1$, 
$$(0v)\alpha = 1(v)\alpha_0\alpha_1\text{ and } (1v)\alpha = 0v,$$ which yields cycles of twice the length of $\alpha_0\alpha_1$'s cycles. Hence $s_n(\alpha) = 2 s_{n-1}(\alpha_0\alpha_1)$.
\end{proof}
We obtain the following immediate corollary.

\begin{corollary}\label{cor:descsettled}
Let $\alpha \in \Omega$ be an element of the form $\alpha=(\alpha_0,\id)\sigma$.
Then $\alpha$ is (strongly) settled if and only if $\alpha_0$ is. Similarly if $\alpha=(\alpha_0,\alpha_1)$, then $\alpha$ is (strongly) settled if and only if $\alpha_0$ and $\alpha_1$ are both (strongly) settled.
\end{corollary}

We next show that the number of stable cycles of an automorphism at any level cannot increase under taking squares.
In fact, a stronger statement holds: for any $\alpha \in \Omega$ and any $m\in \ZZ_2$, the element $\alpha$ is (strongly) settled if and only if $\alpha^m$ is (strongly) settled. A proof of this general statement can be found in \cite{CortezLukina}; here we provide a short proof of the special case needed for our purposes.

\begin{proposition}\label{prop:square}
For any $\alpha \in \Omega$ and $n\geq 1$, $2 s_{n}(\alpha) \leq s_{n+1}(\alpha^2) \leq s_{n+1}(\alpha)$.
Hence the element $\alpha$ is (strongly) settled if and only if $\alpha^2$ is (strongly) settled.
\end{proposition}

\begin{proof}
We first observe that the orbit of a vertex $w \in V$ always has size a power of $2$ and if the $\alpha^2$-orbit of $w$ has more than one element, then $\alpha$ orbit of $w$ has two times the size of the $\alpha^2$-orbit of $w$.

Let $v \in V_n$ be in a stable cycle of $\alpha$ of length $k$. The $\alpha$-orbit of the vertex $v_0 \in V_{n+1}$ lying above $v$ is given by 
\[ \{ (v_0)\alpha, (v_0)\alpha^2,\ldots, (v_0)\alpha^{2k}\}.
\]
Then the $\alpha^2$-orbit of $v_0$ decomposes as the union of two stable orbits:
\[ \{ (v_0)\alpha^2,(v_0)\alpha^4 \ldots, (v_0)\alpha^{2k}\} \text{ and } \{ (v_0)\alpha, (v_0)\alpha^3 \ldots, (v_0)\alpha^{2k-1}\}. \]
Hence, every vertex above $v$ at level $m\geq n+1$ is in a stable cycle of $\alpha^2$ and $2s_n(\alpha) \leq s_{n+1}(\alpha^2)$.

Assume now that $v \in V_n$ is in a stable cycle of $\alpha^2$ of length $k\geq 1$. Then for any $m\geq n+1$ and $v' \in V_{m}$ lying above $v$, the $\alpha^2$-orbit of $v'$ has length $2^{m-n}k$ which is $\geq 2$. Therefore, the $\alpha$-orbit of such $v'$ has length $2^{m-n+1}k$. Since the $\alpha$-orbit of $v$ contains at most $2k$ elements and the $\alpha$-orbit of $v'$ attains the maximum size, $v$ is in a stable cycle of $\alpha$ of length $2k$.

Hence, we find that for all $n\geq 1$, 
\[ 2s_n(\alpha) \leq s_{n+1}(\alpha^2) \leq s_{n+1}(\alpha). \] 
We prove the second claim by taking the limit.
 
\end{proof}

Next, we introduce the notion of a descendant and show that an automorphism is settled if and only if all of its descendants are settled.

\begin{definition}\label{defn:descendant}
    Let $\alpha=(\alpha_0,\alpha_1)\tau \in \Omega$. We define the set of first descendants of $\alpha$ by
   \begin{equation} 
   \Desc_1(\alpha):=\begin{cases} \{\alpha_0,\alpha_1\} & \text{ if }  \tau=\id \\
   \{\alpha_0\alpha_1\} &\text{ if } \tau\neq \id
    \end{cases} \end{equation}

Inductively, for $n \ge 1$, a first descendant of an element of $\Desc_{n-1}(\alpha)$ is called an $n$th descendant of $\alpha$. That is,

\begin{equation}
\Desc_n(\alpha)= \bigcup _{\beta \in \Desc_{n-1}(\alpha)}\Desc_1{(\beta)}
\end{equation}

\end{definition}

The following result will be used in the final section.
\begin{theorem}\label{thm:descsettled} Let $\alpha \in \Omega$. Then the following are equivalent.
\begin{enumerate}
   \item $\alpha$ is (strongly) settled. 
   \item  For all $n\geq 1$, all $n$th descendants of $\alpha$ are (strongly) settled.
    \item There exists $n \ge 1$ such that all $n$th descendants of $\alpha$ are (strongly) settled.
   \end{enumerate}
\end{theorem}
\begin{proof}
If $\alpha = (\alpha_0, \alpha_1)$, then by Proposition~\ref{prop:settledsections},
\[
\frac{s_n(\alpha)}{2^n} =  \frac{1}{2} \left(\frac{s_{n-1}(\alpha_0)}{2^{n-1}}  + \frac{s_{n-1}(\alpha_1)}{2^{n-1}}\right).
\]
Since $s_n(\alpha)/2^n$ is bounded above by $1$, 
$$\lim_{n\to \infty} \frac{s_n(\alpha)}{2^n} = 1 \text{ if and only if } \lim_{n\to \infty} \frac{s_{n-1}(\alpha_0)}{2^{n-1}}=1 \text{ and } \lim_{n\to \infty} \frac{s_{n-1}(\alpha_1)}{2^{n-1}}=1.$$
In particular, for $n\geq 1$, $s_{n}(\alpha)/2^n=1$ if and only if $s_{n-1}(\alpha_0)/2^{n-1}=1$ and $s_{n-1}(\alpha_1)/2^{n-1}=1$.

If $\alpha = (\alpha_0, \alpha_1)\sigma$, then it is conjugate to $(\alpha_0 \alpha_1, \id)\sigma$. By Proposition~\ref{prop:conjsettled}, we may assume that $\alpha=(\alpha_0\alpha_1,\id)\sigma$. By Corollary~\ref{cor:descsettled}, $\alpha$ is (strongly) settled if and only if $\alpha_0\alpha_1$ is (strongly) settled. Hence (1) implies (2), which follows by induction on $n\geq 1$. (2) implies (3) is trivial. 

Now, we show that (3) implies (1). Assume there exists $n\geq 1$ such that all $n$th descendants of $\alpha$ are settled. As we discussed above Corollary~\ref{cor:descsettled} shows that if the first descendants of $\alpha$ are settled, then $\alpha$ is settled. The general statement follows by applying this principle to every $m$th descendant of $\alpha$ for $m\leq n$. 

\end{proof}


\section{The settled elements in the normalizer of the standard odometer}
In this section, we describe the elements in the normalizer of an odometer. Since any two odometers are conjugate, their normalizers are also conjugate in $\Omega$. Therefore, we work with a fixed odometer $\gamma \in \Omega$.

\subsection{The normalizer of an odometer}

Let $\gamma=(\gamma,\id)\sigma \in\Omega$ be the standard odometer and, for each
$k\in\ZZ_2^\times$, the element $z_k\in\Omega$ recursively by
\[
z_k:=(z_k,\gamma^\ell z_k).
\]
where $\ell=\frac{k-1}{2}$. Note that $v_2(k)=0$ implies $\ell \in \ZZ_2$. 
\begin{lemma}
    For any $k\in \ZZ_2^{\times}$, $z_k \gamma z_k^{-1}=\gamma^k$.
\end{lemma}

\begin{proof} 
It is enough to prove that $$z_k  \gamma z_k^{-1} \restr{T_n}=\gamma^k \restr{T_n}$$ for all $n\geq 1$.
We prove the statement by induction on $n$:

For $n=1$, it is immediate that  $z_k\gamma z_k^{-1}\restr{T_1}= \sigma =\gamma ^k\restr{T_1}$.
For $n\geq 1$, 
 $$\begin {array}{lll}
 
 {z_k\gamma z_k^{-1}}\restr{T_n} & = (z_k, \gamma ^{\ell} z_k)(\gamma, \id)\sigma ({z_k}^{-1},(\gamma ^{\ell} z_k)^{-1}) \restr{T_n}\\
  & =(z_k\gamma {z_k}^{-1}(\gamma ^{\ell})^{-1} \restr {T_{n-1}}, \gamma ^{\ell}\restr {T_{n-1}}) \sigma\\
   & =(\gamma ^{k-\ell} \restr {T_{n-1}}, \gamma ^{\ell}\restr {T_{n-1}}) \sigma\\
   & =(\gamma ^{\frac{k+1}{2}} \restr {T_{n-1}}, \gamma ^{\frac{k-1}{2}}\restr {T_{n-1}}) \sigma\\
   & =\gamma ^{k} \restr {T_{n}}. 
 \end{array}$$
\end{proof}

\begin{lemma}\label{lem:signzk}
Let $k\in\ZZ_2^{\times}$. Let $\ell=\frac{k-1}{2}$.
\begin{enumerate}
\item If $v_2(\ell)\ge 1$, then $\sgn_n(z_k)=1$ for all $n\ge 1$.
\item If $v_2(\ell)=0$, then $\sgn_n(z_k)=-1$ for all $n\geq 2$
and $\sgn_1(z_k)=1$.
\end{enumerate}
\end{lemma}
\begin{proof}
For $n=1$, $\sgn_1((z_k,\gamma^{\ell}z_k))=\sgn(\id)=1$. For $n\geq 2$, it follows from the formula given in Remark~\ref{sgndecomp}.
\begin{equation}
 \begin{split}
    \sgn_n(z_k)&=\sgn_n((z_k,\gamma^{\ell}z_k))=\sgn_{n-1}(z_k)^2\sgn_{n-1}(\gamma^{\ell})\\
               &=\sgn_{n-1}(\gamma^{\ell}) \\
    \end{split}
\end{equation}  
If $v_2(\ell)\geq 1$, then $\sgn_n(\gamma^{\ell})=1$ for all $n\geq 2$ and if $v_2(\ell)=0$, then $\sgn_n(\gamma^{\ell})=-1$ for all $n\geq 2$.
\end{proof}

\begin{remark}\label{rk:odosettled}
If $\alpha\in\Omega$ is an odometer, then by definition it acts as a
$2^n$-cycle on each level. Thus $\alpha$ is strongly settled, and every
vertex at level $n\geq 0$ belongs to a stable cycle of $\alpha$. 

For any $k\in\ZZ_2^{\times}$,
\[
z_k^2=(z_k^2,\gamma^\ell z_k\,\gamma^\ell z_k)
       =(z_k^2,\gamma^{\ell(k+1)}z_k^2).
\]
Since $\ell(k+1)=\frac{k^2-1}{2}$, we find
$z_k^2=z_{k^2}$.
\end{remark}

\begin{lemma}\label{lem:zkcongruence}
Let $n\geq 1$ and let $k,k'\in \ZZ_2^\times$ with $v_2(k-k') \geq n$. Then
\[
z_k\restr{T_n} = z_{k'}\restr{T_n}.
\]
\end{lemma}

\begin{proof}
We prove it by induction on $n$.

By definition, $z_k\restr{T_1}=z_{k'}\restr{T_1}=\id$. Suppose the statement holds for some $n\geq1$, and let 
$v_2(k-k')\geq n+1$. Hence, the inductive hypothesis gives
\[
z_k\restr{T_n} = z_{k'}\restr{T_n}. \tag{$\ast$}
\]
Since $k,k'$ are odd and $v_2(k-k')\geq n+1$, we have
\[
v_2((\frac{k-1}{2})-(\frac{k'-1}{2})) \geq n.
\]
As $\gamma\restr{T_n}$ generates a cyclic group of order $2^n$, this congruence gives
\[
\gamma^{(k-1)/2}\restr{T_n} = \gamma^{(k'-1)/2}\restr{T_n}. \tag{$\ast\ast$}
\]
Combining $(\ast)$ and $(\ast\ast)$,
\[ \tag{$\ast\ast\ast$}
(\gamma^{(k-1)/2}z_k)\restr{T_n}
= (\gamma^{(k'-1)/2}z_{k'})\restr{T_n}.
\]
Since $z_k=(z_k,\gamma^{(k-1)/2}z_k)$ and $z_{k'}=(z_{k'},\gamma^{(k'-1)/2}z_{k'})$, using $(\ast)$ and $(\ast\ast\ast)$, we show that $z_k\restr{T_{n+1}}=z_{k'}\restr{T_{n+1}}$, completing the induction.
\end{proof}

\begin{proposition}\label{normalizer of gamma}
The normalizer of the subgroup (topologically) generated by $\gamma$ is
\[ \{ \gamma^mz_k \mid m\in \ZZ_2 \text{ and } k \in \ZZ_2^{\times} \}. \]
\end{proposition}
\begin{proof}
An automorphism of the closed group $\langle\langle \gamma \rangle\rangle$ takes $\gamma$ to $\gamma^k$ for some $k \in \ZZ_2^{\times}$. Hence we have a surjective  map from the normalizer of this group in $\Omega$ to $\ZZ_2^{\times}$. Since, for every $n$,
$\gamma|_{T_n}$ is a $2^n$-cycle as a permutation in $\text{Sym}(2^n)$, its centralizer is the cyclic subgroup it generates in $\text{Sym}(2^n)$, hence also in $\Omega_n$. Passing to the inverse limit, the centralizer of
$\langle\langle\gamma\rangle\rangle$ in $\Omega$ is
$\langle\langle\gamma\rangle\rangle$ itself.

For all $k\in \ZZ_2^{\times}$, the elements $z_k$ are in the normalizer, hence the normalizer is  $\{ \gamma^mz_k \mid m\in \ZZ_2 \text{ and } k \in \ZZ_2^{\times} \}.$
\end{proof}

We now analyze the settled structure of every element $\alpha \in N(\gamma)$.

\begin{definition}
An element $\alpha \in \Omega$ is \emph{uniformly settled} if there exists a constant $C > 0$ such that, for all $n \geq 1$, there are at most $C$ vertices at level $n$ that do not lie in a stable cycle of $\alpha$.
\end{definition}

We will show that every settled element of $N(\gamma)$ is uniformly
settled. The analysis splits into the two cases
$v_2(\frac{k-1}{2})=0$ and $v_2(\frac{k-1}{2})\geq 1$.
From now on, set $\ell=\frac{k-1}{2}$.

We use Lemma~\ref{mkidentitylevel} to prove Lemma~\ref{kidentitylevel}. 

\begin{lemma}\label{mkidentitylevel}
    Let $k\in \ZZ_2^{\times}$ and $m\in \ZZ_2$. If $v_2(m)\leq v_2(\ell)$, then the automorphism $\gamma^{m}z_k \restr{T_n}=\id$, for all $n\leq v_2(m)$.
\end{lemma}
\begin{proof}
We'll prove it by induction on $n$. If $v_2(m)=0$ the claim is immediate. Assume $v_2(m)\ge 1$ and let
$n\le v_2(m)$. Then
\[
\gamma^m z_k\restr{T_n}
  =(\gamma^{\frac{m}{2}}z_k\restr{T_{n-1}},\,
    \gamma^{\frac{m}{2}+\ell}z_k\restr{T_{n-1}}).
\]
Since $v_2(\frac{m}{2})=v_2(m)-1<v_2(m)\le v_2(\ell)$,
we have $v_2(\frac{m}{2}+\ell)=v_2(m)-1$. The induction hypothesis gives that
both sections act trivially on $T_{n-1}$, hence
$\gamma^m z_k\restr{T_n}=\id$.
\end{proof}

\begin{lemma}\label{kidentitylevel}
    The automorphism $z_k \restr{T_n}$ is the identity for all $n\leq v_2(\ell)+1$.
\end{lemma}
\begin{proof}
    For $n=1$, the statement holds trivially since $z_k\restr{T_1}=\id$. Let $n\leq v_2(\ell)+1$ and assume that $z_k\restr{T_{n-1}}=\id$.

    We know that $$z_k \restr{T_n}=(z_k\restr{T_{n-1}}, \gamma^{\ell}z_k\restr{T_{n-1}}).$$
 By Lemma~\ref{mkidentitylevel}, $\gamma^{\ell}z_k \restr{T_{v_2(\ell)}}=\id$. Combining this with our hypothesis, $z_k\restr{T_n}=\id$.
\end{proof}

We will prove the following theorem in a sequence of propositions.
\begin{theorem}\label{thm:classificationzk}
Let $m\in \ZZ_2$ and $k \neq \pm 1 \in \ZZ_2^{\times}$.
\begin{enumerate}
    \item If $v_2(\ell) \geq 1$, then $\gamma^m z_k$ is uniformly settled and $$s_n(\gamma^m z_k) \geq 2^n-2^{v_2(\ell)}$$ for all $n\geq v_2(\ell)$.
    \item If $v_2(\ell)=0$, then $\gamma^m z_k$ is uniformly settled and $$s_n(\gamma^m z_k) \geq 2^n-2^{v_2(k+1)}$$ for all $n\geq v_2(k+1)$.
\end{enumerate} 
\end{theorem}

We first state an important corollary of Theorem~\ref{thm:classificationzk}.

\begin{corollary}\label{zksummary}
Let $k \neq \pm 1 \in \ZZ_2^{\times}$ and $m\in \ZZ_2$. Let $\nu=v_2(\frac{k^2-1}{4})$.
   \begin{enumerate}
   \item The automorphism $\gamma^m z_k$ is uniformly settled.
   \item For any $n \geq \nu + 1$, the number of vertices that are not in a stable cycle of $\gamma^mz_k$ at level $n$ is at most $2^{\nu +1}$. 
     \end{enumerate}
\end{corollary}
\begin{proof}
   We first observe that if $v_2(\ell)\ge 1$, then $
v_2\!\left(\frac{k+1}{2}\right)=0,
$
while if $v_2(\ell)=0$, then 
$
v_2\!\left(\frac{k+1}{2}\right)\ge 1.
$
Thus, in each case, the 2-adic valuation of $\frac{k^2-1}{4}$ coincides with one of the valuations
$
v_2\!\left(\frac{k-1}{2}\right)
\; \text{or} \;
v_2\!\left(\frac{k+1}{2}\right).
$
The statement follows now from applying Theorem~\ref{thm:classificationzk}. 

\end{proof}

\subsection{Proof of Theorem~\ref{thm:classificationzk}} \ \\

\textbf{Case $1$: $v_2(\ell)\ge 1$}.
\begin{proposition}\label{prop1:liseven}
Let $m,s\in\ZZ_2$ and $k\ne \pm 1\in\ZZ_2^{\times}$. Assume $v_2(\ell)\geq 1$.
\begin{enumerate}
\item If $v_2(m)=v_2(s)\le v_2(\ell)$, then $\gamma^m z_k$ and
$\gamma^s z_k$ are conjugate.
\item If $v_2(m)\le v_2(\ell)$, then $\gamma^m z_k$ is strongly settled.
Moreover every vertex at level $n\ge v_2(m)$ lies in a stable cycle of
$\gamma^m z_k$, i.e.\ $s_n(\gamma^m z_k)=2^n$.
\end{enumerate}
\end{proposition}
\begin{proof}

If $v_2(m)=v_2(s)=0$, then, using Lemma~\ref{lem:signzk}, we obtain  
\[ \sgn_n(\gamma^m z_k)=\sgn_n(\gamma^s z_k)=\sgn_n(\gamma)=-1 \]
for all $n\geq 1$.  
Hence $\gamma^m z_k$ and $\gamma^s z_k$ are both odometers by Proposition~\ref{prop:odosgn}. Since any two odometers are conjugate in $\Omega$, it follows that $\gamma^m z_k$ and $\gamma^s z_k$ are conjugate. Moreover, they are strongly settled by Remark~\ref{rk:odosettled}.

Thus, we may assume $v_2(m)\ge 1$. We compute:
\[
\begin{split}
\gamma^m z_k 
&= (\gamma^{\frac{m}{2}},\gamma^{\frac{m}{2}})(z_k, \gamma^{\ell}z_k)\\
&=(\gamma^{\frac{m}{2}} z_k,\gamma^{\frac{m}{2}+\ell} z_k).
\end{split}
\]

Assume now that $1\le v_2(m)=v_2(s)=t\le v_2(\ell)$.  
Assume inductively that for any $r,q\in \mathbb{Z}_2$ with $v_2(r)=v_2(q)<t$, the automorphism $\gamma^r z_k$ is strongly settled and the elements $\gamma^r z_k$ and $\gamma^q z_k$ are conjugate.

By assumption $v_2(m)\le v_2(\ell)$, and therefore $v_2(\frac{m}{2})<v_2(\ell)$. This implies that
\[
v_2\!\left(\frac{m}{2}+\ell\right)
=\min\{\,v_2(\tfrac{m}{2}),\,v_2(\ell)\,\}
= v_2(\tfrac{m}{2})<t.
\]
By the induction hypothesis, the elements $\gamma^{\frac{m}{2}} z_k$ and $\gamma^{\frac{m}{2}+\ell} z_k$ are both strongly settled, and they are conjugate respectively to $\gamma^{s/2} z_k$ and $\gamma^{s/2+\ell} z_k$. By Theorem~\ref{thm:descsettled}, $\gamma^m z_k$ is strongly settled, and by Proposition~\ref{prop:conj}, it is conjugate to $\gamma^s z_k$.

By Proposition~\ref{mkidentitylevel}, the automorphism $\gamma^m z_k$ acts as the identity on level $v_2(m)$. There are $2^{v_2(m)}$ subtrees at this level. Writing $m=2^{v_2(m)} m'$, the element $\gamma^m$ acts as $\gamma^{m'}$ on each subtree, while $z_k$ acts as $\gamma^i z_k$ with either $i=0$ or $v_2(i)\ge 1$. Thus $\gamma^m z_k$ acts by $\gamma^{m'+i} z_k$ on each subtree, and this is always an odometer since $m'+i \in \mathbb{Z}_2^\times$.

Hence every vertex on level $v_2(m)$ lies in a stable cycle of $\gamma^m z_k$.

\end{proof}

\begin{proposition}\label{prop2:liseven}
Let $m\in\ZZ_2$, $k\ne 1\in\ZZ_2^{\times}$, and assume that
$v_2(\ell)\ge 1$. If $v_2(m)>v_2(\ell)$, then:
\begin{enumerate}
\item $\gamma^m z_k$ is conjugate to $z_k$.
\item For all $n\ge v_2(\ell)$, the number of vertices at level $n$ not lying in
a stable cycle of $\gamma^m z_k$ equals $2^{v_2(\ell)}$, i.e.
\[
s_n(\gamma^mz_k)=2^n-2^{v_2(\ell)}.
\]
\end{enumerate}
\end{proposition}
\begin{proof}

It is enough to show that $\gamma^m z_k \restr{T_n}$ is conjugate to $z_k\restr{T_n}$ for all $n\geq 1$. We proceed by induction on $n$. Assume for induction that the statement holds at level $n-1$, 
that is, for any $t$ with $v_2(t)>v_2(\ell)$, 
$\gamma^t z_k\restr{T_{n-1}}$ is conjugate to $z_k\restr{T_{n-1}}$.

Assume $v_2(\tfrac{m}{2})>v_2(\ell)$. In this case, $v_2(\tfrac{m}{2}+\ell)=v_2(\ell)$. By Proposition~\ref{prop1:liseven}, the automorphism 
$\gamma^{\frac{m}{2}+\ell}z_k \restr{T_{n-1}}$ is conjugate to 
$\gamma^{\ell}z_k\restr{T_{n-1}}$. 
By the induction hypothesis, 
$\gamma^{\frac{m}{2}}z_k\restr{T_{n-1}}$ is conjugate to $z_k\restr{T_{n-1}}$. Hence, by Proposition~\ref{prop:conj}, 
$\gamma^m z_k=(\gamma^{\frac{m}{2}} z_k,\gamma^{\frac{m}{2}+ \ell}z_k)\restr{T_n}$ 
is conjugate to 
$z_k=(z_k, \gamma^{\ell}z_k)\restr{T_n}$.

Assume now that $v_2(\tfrac{m}{2})=v_2(\ell)$. In this case, $v_2(\tfrac{m}{2}+\ell)>v_2(\ell)$, so 
$\gamma^{\frac{m}{2}+\ell}z_k \restr{T_{n-1}}$ is conjugate to 
$z_k\restr{T_{n-1}}$ by the induction hypothesis, 
and $\gamma^{\frac{m}{2}}z_k\restr{T_{n-1}}$ is conjugate to 
$\gamma^{\ell}z_k\restr{T_{n-1}}$ by Proposition~\ref{prop1:liseven}. 
Thus, Proposition~\ref{prop:conj} again finishes the statement.

By Proposition~\ref{prop:conjsettled}, it suffices to prove the statement for $z_k$. We now prove by induction on $n$ that 
$s_n(z_k)=2^n-2^{v_2(\ell)}$ for all $n \geq v_2(\ell)$. By Lemma~\ref{kidentitylevel}, every vertex at level $v_2(\ell)+1$ is fixed by $z_k$, which implies that all of the vertices at level $v_2(\ell)$ are also fixed, hence $z_k$ has no stable cycles at level $v_2(\ell)+1$ and therefore $s_{v_2(\ell)}(z_k)=0$.

Assume by induction that 
$s_m(z_k)= 2^m-2^{v_2(\ell)}$ for all $m$ with $v_2(\ell) \leq m<n$. 
By Proposition~\ref{prop:settledsections},  
$$ s_n(z_k)=s_{n-1}(z_k) + s_{n-1}(\gamma^{\ell} z_k).$$ 
Since $n-1 \geq v_2(\ell)$, Proposition~\ref{prop1:liseven} implies 
$s_{n-1}(\gamma^{\ell} z_k)=2^{n-1}$. 
Therefore, 
$$ s_n(z_k)=s_{n-1}(z_k) + 2^{n-1}.$$ 
By the induction assumption, 
$$s_n(z_k)= 2^{n-1} -2^{v_2(\ell)}+2^{n-1}=2^n-2^{v_2(\ell)}.$$

\end{proof}

\textbf{Case $2$: $v_2(\ell)=0$.}
\begin{proposition}\label{prop1:lisodd} 
Let $m\in \ZZ_2$ and $k\in \ZZ_2^\times$ with $k\neq -1$ and $v_2(\ell)=0$, where $\ell=(k-1)/2$.  
Assume $v_2(m)=0$. Then:
\begin{enumerate}
    \item The automorphism $\gamma^m z_k$ is conjugate to  $\gamma^{m'} z_k$  for every $m'$ with $v_2(m')=0$.
    \item Every vertex of $T_n$ lies in a stable cycle of  $\gamma^m z_k$  for all $n\geq v_2(k+1)$. 
    Hence $\gamma^m z_k$ is strongly settled.
\end{enumerate}
\end{proposition}
\begin{proof}
Assume $v_2(m)=0$. Then
\[
\begin{split}
\gamma^m z_k 
   &= (\gamma^{\frac{m+1}{2}},\gamma^{\frac{m-1}{2}})\sigma (z_k,\gamma^\ell z_k) \\
   &= (\gamma^{\frac{m+1}{2}}\gamma^\ell z_k,\gamma^{\frac{m-1}{2}}z_k)\sigma .
\end{split}
\]
By Proposition~\ref{prop:conj}, $\gamma^m z_k$ is conjugate to
\[
(\gamma^{\frac{m(k+1)}{2}} z_{k^2}, \id)\sigma .
\]

We compute:
\[
v_2\!\left(\frac{m(k+1)}{2}\right)
   = v_2(k+1)-1
   \le v_2(k+1)
   = v_2\!\left(\frac{k^2-1}{2}\right).
\]
Since $v_2(k^2-1)\ge 2$, Proposition~\ref{prop1:liseven} applies to 
$\gamma^{\frac{m(k+1)}{2}} z_{k^2}$ and shows that it is strongly settled with
\[
s_n\!\left(\gamma^{\frac{m(k+1)}{2}} z_{k^2}\right) = 2^n
\qquad\text{for all } n\ge v_2(k+1)-1 .
\]

By Proposition~\ref{prop:conjsettled},
\[
s_n(\gamma^m z_k)
   = s_n\big((\gamma^{\frac{m(k+1)}{2}} z_{k^2},\id)\sigma\big),
\]
and Proposition~\ref{prop:settledsections} gives
\[
s_{n+1}(\gamma^m z_k)
   = 2\, s_n\!\left(\gamma^{\frac{m(k+1)}{2}}z_{k^2}\right)
   = 2^{n+1}
\]
for all $n\ge v_2(k+1)-1$.  
This proves the second claim.

\medskip

To prove (1) we show that $\gamma^m z_k$ is conjugate to $\gamma z_k$.  
From the computation above, $\gamma^m z_k$ is conjugate to  
\[
(\gamma^{\frac{m(k+1)}{2}} z_{k^2},\id)\sigma,
\]
and similarly,
$\gamma z_k $ is conjugate to
$(\gamma^{\frac{k+1}{2}} z_{k^2},\id)\sigma $. Thus it suffices to show that
$\gamma^{\frac{m(k+1)}{2}} z_{k^2}$
and
$\gamma^{\frac{k+1}{2}} z_{k^2}$ are conjugate. Finally, both exponents satisfy
$v_2(\frac{m(k+1)}{2})
   = v_2(k+1)-1
   = v_2(\frac{k+1}{2})$,
and by Proposition~\ref{prop1:liseven},  
$\gamma^t z_{k^2}$ depends up to conjugacy only on $v_2(t)$.  
Hence the two automorphisms are conjugate, completing the proof.
\end{proof}

\begin{proposition}\label{prop2:lisodd} Let $m\in \ZZ_2$ and $k \neq -1 \in \ZZ_2^{\times}$ such that $v_2(\ell)=0$.
   Assume $v_2(m) \geq 1$.
   \begin{enumerate}
       \item The automorphism $\gamma^mz_k$ is conjugate to $z_k$.
       \item The number of vertices at level $n \geq v_2(k+1)$ that do not belong to a stable cycle of $\gamma^m z_k$ is less than or equal to $2^{v_2(k+1)}$, that is $s_n(z_k) \geq 2^n-2^{v_2(k+1)}$.
    \end{enumerate}
\end{proposition}
\begin{proof}
Assume $v_2(m)\ge 1$. We show that $\gamma^m z_k\restr{T_n}$ is conjugate to $z_k\restr{T_n}$ for all $n\ge1$.  
Fix $n\ge1$ and assume inductively that for every $r\in\ZZ_2$ with $v_2(r)\ge1$, the restriction $\gamma^r z_k\restr{T_{n-1}}$ is conjugate to $z_k\restr{T_{n-1}}$.

Assume $v_2(m)=1$. 
Then $v_2(m/2+\ell)\ge1$, so by the induction hypothesis  
$\gamma^{m/2+\ell}z_k\restr{T_{n-1}}\sim z_k\restr{T_{n-1}}$,  
and by Proposition~\ref{prop1:lisodd},  
$\gamma^{m/2}z_k\restr{T_{n-1}}\sim \gamma^\ell z_k\restr{T_{n-1}}$.  
Proposition~\ref{prop:conj} then implies
\[
\gamma^m z_k=(\gamma^{m/2}z_k,\ \gamma^{m/2+\ell}z_k)\restr{T_n}
   \sim (z_k,\gamma^\ell z_k)\restr{T_n}=z_k \restr{T_n}.
\]

Assume now $v_2(m)>1$.  
Here $v_2(m/2+\ell)=0$, so Proposition~\ref{prop1:lisodd} gives  
$\gamma^{m/2+\ell}z_k\restr{T_{n-1}}\sim\gamma^\ell z_k\restr{T_{n-1}}$, and the induction hypothesis gives  
$\gamma^{m/2}z_k\restr{T_{n-1}}\sim z_k\restr{T_{n-1}}$.  
Again Proposition~\ref{prop:conj} yields  
\[
\gamma^m z_k\restr{T_n}\sim z_k\restr{T_n}.
\]

Thus $\gamma^m z_k$ is conjugate to $z_k$ on every $T_n$.  
By Proposition~\ref{prop:conjsettled}, it suffices to compute $s_n(z_k)$.

\smallskip
By Proposition~\ref{prop:square}, 
\[
s_{n}(z_k^2) \leq s_n(z_k).
\] and by Remark~\ref{rk:odosettled}, $s_{n}(z_k^2)=s_n(z_{k^2})$.
Because $v_2((k^2-1)/2)=v_2(k+1)\ge 1$, Proposition~\ref{prop2:liseven} yields that for all $n \geq v_2(k+1)$, 
\[
2^n-s_{n}(z_{k}) \leq 2^n-s_{n}(z_{k}^2) \leq 2^{v_2(k+1)}.
\]

\end{proof}

\subsection*{Proof of Theorem~\ref{thm:classificationzk}}
If $v_2(\frac{k-1}{2}) \geq 1$, then it follows from Proposition~\ref{prop1:liseven} and \ref{prop2:liseven}. Similarly, if $v_2(\frac{k-1}{2})=0$, then the result follows from Proposition~\ref{prop1:lisodd} and ~\ref{prop2:lisodd}.

\section{Arithmetic monodromy groups of quadratic polynomials with a periodic postcritical orbit}

For $r \geq 2$, we recursively define a set of $r$ elements in $\Omega$ as follows: 
\begin{equation}\label{generators}
    \begin{split}
     a_1=&(a_r,\id)\sigma \\ 
     a_2=&(a_{1},\id) \\
        & \vdots \\
      a_r = &(a_{r-1},\id)   
    \end{split}
\end{equation}
Here, $a_1$ acts nontrivially on level $1$, while $a_i$ for $1<i\le r$ act trivially on level $1$. For $n\ge2$, the restriction $a_i\restr{T_n}$ is defined recursively in terms of $a_{i-1}\restr{T_{n-1}}$.

Let $G$ be the topological closure of the subgroup of $\Omega$ generated by $a_1,\ldots,a_r$. This group is studied in detail by Pink \cite{Pinkpolyn}. The profinite geometric iterated monodromy group of a quadratic polynomial with a periodic postcritical orbit of length $r$ is a subgroup of $\Omega$ and it is conjugate to $G$; see \cite[Theorem~2.4.1]{Pinkpolyn}

We first show that $G$ contains the standard odometer $\gamma$.

\begin{proposition}
  The odometer $\gamma=(\gamma,\id)\sigma$ is an element of $G$. 
\end{proposition}
\begin{proof}
    We show that for any $n\geq 1$ there exists an element $g_n \in G$ such that
    \[ \gamma \restr{T_n}= g_n(a_1\cdots a_r) g_n^{-1} \restr{T_n}.\]
  For $n=1$, it is clear since $\gamma\restr{T_1}=a_1 \ldots a_r \restr{T_1}=\sigma$. Assume that there is $g_{n-1} \in G$ such that 
  \[ \gamma \restr{T_{n-1}}= g_{n-1}(a_1 \cdots a_r) g_{n-1}^{-1} \restr{T_{n-1}}.\]
Let $ \alpha=(a_1 \ldots a_r, \id)\sigma$. Then
\begin{equation}
    \begin{split}
        \alpha=&(a_1,\id)(a_2,\id)\ldots (a_{r-1},\id) (a_r,\id)\sigma  \\
        =& a_2a_3 \ldots a_r a_1 \\
        =&a_1^{-1}(a_1a_2 \cdots a_r)a_1
    \end{split}
\end{equation}

Since $\alpha$ is conjugate to $a_1 a_2\cdots a_r$ in $G$, it suffices to show that $\alpha \restr{T_n}$ and $\gamma \restr{T_n}$ are conjugate in $G_{n}$.  Now we have
\begin{equation}
    \begin{split}
        \alpha\restr{T_n}&=(a_1 a_2\ldots a_{r}\restr{T_{n-1}},\id)\sigma \\
                    &=(g_{n-1}^{-1}\gamma g_{n-1} \restr{T_{n-1}}, \id)\sigma \\
                    &=(g_{n-1}^{-1}, g_{n-1}^{-1})\gamma (g_{n-1},g_{n-1}) \restr{T_n}.
    \end{split}
\end{equation}
By Lemma~\ref{lem:diagonal}, $(g_{n-1},g_{n-1})$ is an element of $G$ and the statement follows. 
\end{proof}

\begin{lemma}\label{lem:diagonal}
    For any $u \in G$, the element $(u,u)$ also lies in $G$.
\end{lemma}
\begin{proof}
For every $i=1,\ldots ,r$, we have $(a_i,a_i) \in G$: $a_1^2=(a_r,a_r) \in G$ and
for $i=2,\ldots,r$, the product $a_ia_1a_ia_1^{-1}$ gives the desired element.
\end{proof}

\subsection{Summary of Pink's results on the normalizer of $G$ in $\Omega$}

We let $N(G)$ denote the normalizer of $G$ in $\Omega$ and $N(\gamma)$ the normalizer of the topological group generated by $\gamma$ in $\Omega$.

Pink \cite{Pinkpolyn} studies the action of $N(G)$ on the maximal abelian quotient $G_{\text{ab}}$ of $G$. A standard group-theoretic argument shows that this action factors through $N(G)/G$. Pink shows that the kernel of this action is exactly $G$. Moreover, $N(G)/G$ embeds into the automorphism group of $G_{\text{ab}}$, which is isomorphic to the automorphism group of the direct product of $r$ copies of $\ZZ_2$.
The image of this embedding is the direct product of $r$ copies of the automorphism group of $\ZZ_2$ which is isomorphic to $\ZZ_2^{\times}$.

\begin{theorem}[{\cite[Theorem 2.6.8]{Pinkpolyn}}]\label{PinkNorm1}
Let $r \ge 2$, and let $N(G)$ denote the normalizer of $G$ in $\Omega$. Then
\begin{equation}\label{eq:NG/G}
    N(G)/G \cong (\mathbb{Z}_2^{\times})^{r}.
\end{equation}
\end{theorem}

The next theorem identifies $N(\gamma)$ as a subgroup of $N(G)$ and describes its image under the isomorphism given in ~\eqref{eq:NG/G}.
In particular, the quotient $N(\gamma)G/G$ embeds into the diagonal subgroup of $(\ZZ_2^{\times})^r$. Moreover, for any $\tau \in \ga$, the image of $\tau$ is also a diagonal element in $(\ZZ_2^{\times})^r$.

\begin{theorem}[{\cite[Theorem 2.7.7 and Theorem~2.8.4]{Pinkpolyn}}]\label{PinkNorm2}\hfill
\begin{enumerate}

    \item
    Let $k \in \mathbb{Z}_2^{\times}$ and let $z \in \Omega$ such that  $z\gamma z^{-1} = \gamma^{k}.$
    Then $z \in N(G)$, and the image of the coset $zG$ under the isomorphism given in ~\eqref{eq:NG/G} is the diagonal element
    \[
        (k, \ldots, k) \in (\mathbb{Z}_2^{\times})^{r}.
    \]

    \item
    Let $\tau \in \Gal(\bar{K}/K)$. Under the composite homomorphism
    \[
        \Gal(\bar{K}/K)
            \longrightarrow \ga / G
            \hookrightarrow N(G)/G
            \xrightarrow{\;\sim\;} (\mathbb{Z}_2^{\times})^{r},
    \]
    the element $\tau$ maps to $(\mu,\ldots,\mu)$, where $\mu \in \mathbb{Z}_2^{\times}$ is determined by
    \[
        \tau(\zeta) = \zeta^{\mu}
        \qquad \text{for every $2^{m}$th root of unity $\zeta$}.
    \]
    Moreover, if $K$ is finitely generated over its prime field, then the image of $\Gal(\bar{K}/K)$ in $\mathbb{Z}_2^{\times}$ is of finite index.

\end{enumerate}
\end{theorem}

Let $ \mathscr{G}$ denote the subset of $\Omega$ defined as
\begin{equation*}
    \mathscr{G}:=G N(\gamma)
\end{equation*}

\begin{proposition}\label{Gar has finite index}\hfill
    \begin{enumerate}
        \item $\mathscr{G}$ is a subgroup of $N(G)$ and it contains $\ga$.
        \item If $K$ is finitely generated over its prime field, then $\ga$ has finite index in $\mathscr{G}$.
    \end{enumerate}
\end{proposition}
\begin{proof}
Since $N(\gamma)$ is contained in the normalizer $N(G)$ of $G$, $\mathscr{G}$ is a subgroup of $N(G)$. 

We note that the map $\Gal(\bar{K}/K) \to \ga/G$ is surjective. Let $\alpha \in \ga$. Then there is some $\tau \in \Gal(\bar{K}/K)$ such that the images of $\tau$ and the coset $G\alpha$ agree in $(\ZZ_2^{\times})^{r}$. By the second part of Theorem~\ref{PinkNorm2}, this image is $(k,\ldots,k)$ for some $k \in \ZZ_2^{\times}$. Proposition~\ref{normalizer of gamma} and the first part of Theorem~\ref{PinkNorm2} imply that $G\alpha=Gz_k$ in $N(G)/G$. Since $\mathscr{G}$ contains $N(\gamma)$ and $z_k \in N(\gamma)$, we have that $\alpha \in \mathscr{G}$. 
 
Assume now that $K$ is finitely generated over its prime field and let $\beta \in \mathscr{G}$, then $\beta \ga= z_k\ga$ for some $k \in \ZZ_2^{\times}$. By the last part of Theorem~\ref{PinkNorm2}, there are only finitely many $k \in \ZZ_2^{\times}$ such that $(k,\ldots,k)$ is not in the image of $\ga/G$. Hence $\ga$ is a finite index subgroup of $\mathscr{G}$. 

\end{proof}
Using Proposition~\ref{normalizer of gamma}, we observe that $\mathscr{G}$ is
topologically generated by the set
\[ \{ a_1,\ldots,a_r, z_k \mid k \in \ZZ_2^{\times} \},\]
where $z_k=(z_k,\gamma^{\ell}z_k)$. We will prove that $\mathscr{G}$ contains a dense subset consisting of settled elements. 
\begin{remark}\label{rem:Gar}
    Let $H$ be a subgroup of $\Omega$. If $H$ is densely settled as a subgroup of $\Omega$, then any subgroup of $H$ of finite index is also densely settled. This essentially follows from the fact that a finite index subgroup is open. Hence by Proposition~\ref{Gar has finite index}, if $\mathscr{G}$ is densely settled, and $K$ is finitely generated over its prime field, then $\ga$ is also densely settled.
\end{remark}

\section{Producing settled elements in $\mathscr{G}$}
Theorem~\ref{PinkNorm2} implies that $\mathscr{G}$ is the union of right cosets $Gz_k$ over $k \in \ZZ_2^{\times}$. We claim that each coset
$Gz_k$, for $k\neq \pm 1$ contains sufficiently many settled elements. For example, let $m_1,\ldots,m_r \in \ZZ_2^{\times}$ and consider the product $a_1^{m_1}\ldots a_r^{m_r}$. By Proposition~\ref{prop:odosgn}, it is an odometer in $G$ since $\sgn_n(a_1^{m_1}\ldots a_r^{m_r})=-1$ for all $n\geq 1$. Hence it is settled. However, $a_i$ is not settled for any $i \in \{1,\ldots,r\}$ since it acts identity one one subtree at level one. In Proposition~\ref{prop:azk}, we show that $a_iz_k$ is settled for all $1\leq i \leq r $ and $k \neq \pm 1 \in \ZZ_2^{\times}$.

To exhibit more examples of settled elements in $\mathscr{G}$, we introduce the concept of stable blocks.
Let $\mathscr{U}(\gamma)$ be the subset of $N(\gamma)$ defined by
$$ \mathscr{U}(\gamma):=\{ \gamma^m z_k \quad | \quad  k \neq \pm 1,\; m \in \ZZ_2 \}.
$$
We note that the elements of $\mathscr{U}(\gamma)$ are all uniformly settled by Corollary~\ref{zksummary}. 
\begin{remark}\label{Ugamma}
 Note that if $\alpha \in \mathscr{U}(\gamma)$ then $\alpha=\gamma^mz_k$ for some $m\in \ZZ_2$ and $k\neq \pm 1$. If $m$ is even, then $\alpha =(\gamma^{\frac{m}{2}}z_k, \gamma^{\frac{m}{2}+\ell}z_k)$ and the first descendants of $\alpha$ are $\gamma^{\frac{m}{2}}z_k$ and $\gamma^{\frac{m}{2}+\ell}z_k$. Similarly, if $m \in \ZZ_2^{\times}$, then $$\alpha =(\gamma^{\frac{(m+2\ell+1)}{2}}z_k, \gamma^{\frac{(m-1)}{2}}z_k)\sigma$$  and $\alpha$ has only one first descendant which equals $\gamma^{m+\ell}z_{k^2}$. Hence, in both cases, the first descendants of $\alpha$ are again $\mathscr{U}(\gamma)$.
 
Since $(n+1)st$ descendants of $\alpha$ are the 
 $nth$ descendants of first descendants of $\alpha \in \mathscr{U}(\gamma)$, it follows that for $\alpha \in \mathscr{U}(\gamma)$, all $n$th descendants of $\alpha$ are also in $\mathscr{U}(\gamma)$ for all $n\geq 1$. 
\end{remark} 

\begin{definition}\label{defn:stableblock}
We call an element $\alpha=(\alpha_0,\alpha_1)\tau \in \Omega$ a \emph{block} if $\tau \neq \id$ or  if $\tau=\id$, then at least one of $\alpha_0$ or $\alpha_1$ is in $\mathscr{U}(\gamma)$. We call a block $\alpha \in \Omega$ \emph{stable} if its $n$th descendants are all blocks for any $n\geq 1$.
\end{definition}

From the definition and Remark~\ref{Ugamma}, it immediately follows that the element $\gamma^mz_k$ is a stable block for any $k\neq \pm1 \in \ZZ_2^{\times}$ and any $m \in \ZZ_2$.

\begin{proposition}
    If $\alpha \in \Omega$ is a stable block, then the $n$th descendants of $\alpha$ are also stable blocks for all $n\geq 1$.
\end{proposition}
\begin{proof}
    It is enough to show that the first descendants of $\alpha$ are stable blocks, and the rest follows by induction. If $\beta$ is a descendant of $\alpha$, then it is a block by definition. We need to show that all descendants of $\beta$ are blocks, but these are also descendants of $\alpha$, hence they are also blocks by definition of a stable block.
\end{proof}

The goal of this section is to prove the following key result. 

\begin{theorem}\label{thm:stablesettled}
    Every stable block in $\mathscr{G}$ is settled.
\end{theorem}

For $\alpha \in \Omega$, let $D_n(\alpha)$ be the set of $n$th descendants of $\alpha$ that are not contained in $\mathscr{U}(\gamma)$, that is,
\[ D_n(\alpha):=\{ \beta \in \Desc_n(\alpha) \; | \; \beta \not \in \mathscr{U}(\gamma)\}. \]

\begin{lemma}
     If $\alpha \in \Omega $ is a stable block, then for all $n\geq 1$, $D_n(\alpha)$ has at most one element.
\end{lemma} 
\begin{proof}
    Let $\alpha$ be a stable block. By definition, if there are two elements in $\Desc_1(\alpha)$, then one of them is in $\mathscr{U}(\gamma)$. Hence $|D_1(\alpha)| \leq 1$.
    By Remark~\ref{Ugamma}, the descendants of the elements in $\mathscr{U}(\gamma)$ are again in $\mathscr{U}(\gamma)$. Hence if $D_n(\alpha)=\{\beta \}$ for some $n\geq 1$, then $D_{n+1}(\alpha)$ is contained in $\Desc_1(\beta)$. The statement follows by induction on $n$.
\end{proof}

\subsection{Proof of Theorem~\ref{thm:stablesettled}.}
We will now prove Theorem~\ref{thm:stablesettled} in a sequence of lemmas depending on the behaviour of the set $D_n(\alpha)$. 
\begin{lemma}\label{lem:Dempty}
  Let $\alpha \in \Omega$. If $D_n(\alpha)$ is empty for some $n\geq 1$, then $\alpha$ is settled. 
\end{lemma}
\begin{proof}
   If $D_n(\alpha)$ is empty for some $n\geq 1$, then every element in $\Desc_n(\alpha)$ is settled since they are all in $\mathscr{U}(\gamma)$. Hence by Theorem~\ref{thm:descsettled}, $\alpha$ is settled. 
\end{proof}
We observe that if $D_n(\alpha)$ is empty for some $n\geq 1$, then $D_m(\alpha)=\emptyset$ for all $m\geq n$.

\begin{lemma}\label{lem:Dnsize1}
  Let $\alpha \in \Omega$ and $D_n(\alpha)=\{\beta_n\}$ for all $n\geq 1$. Then $\alpha$ is settled iff $\beta_n$ is settled for some $n\geq 1$.
\end{lemma}
\begin{proof}
    This follows from the fact that $\Desc_n(\alpha)$ consists of elements from $\mathscr{U}(\gamma)$ and $\beta_n$. Hence, the statement follows from Theorem~\ref{thm:descsettled} and the fact that every element of $\mathscr{U}(\gamma)$ is settled.
\end{proof}

For $\alpha \in \Omega$, assume that $D_n(\alpha)=\{\beta_n\}$ as in Lemma~\ref{lem:Dnsize1}. Then we may assign a sequence $(d(\alpha)_n)_{n\geq 0}$ to $\alpha$ as follows: $d(\alpha)_0=\sgn_1(\alpha)$ and 
 $$ d(\alpha)_n:=\sgn_1(\beta_n)$$
 for all $n\geq 1$.

 We note that by Remark~\ref{Ugamma}, the descendants of an element in $\mathscr{U}(\gamma)$ are again in $\mathscr{U}(\gamma)$, and hence $\beta_{n+1}$ is a first descendant of $\beta_n$ for all $n\geq 1$. 

\begin{lemma}\label{lem:-1-1-1}
   Let $\alpha \in \Omega$ and assume that $D_n(\alpha)=\{\beta_n\}$ for all $n\geq 1$.  
   If there is an integer $m\geq 1$ such that $d(\alpha)_n=-1$ for all $n\geq m$, then $\alpha$ is settled.
\end{lemma}
\begin{proof}
Assume $d(\alpha)_n=-1$ for all $n\geq m$. By the definition of $d(\alpha)_n$, $\beta_n\restr{T_1} \neq \id$ and hence $\Desc_1(\beta_n)$ contains only one element, namely $\beta_{n+1}$. Therefore, we obtain using Remark~\ref{sgndecomp} that
   $$\sgn_k(\beta_m)=\sgn_1(\beta_{m+k-1})=-1$$ for all $k\geq 1$. By Proposition~\ref{prop:odosgn}, $\beta_m$ is settled and by Lemma~\ref{lem:Dnsize1}, $\alpha$ is settled. 
   
\end{proof}
To estimate the number of stable cycles, we first need to know which coset the descendants of an element belong to.
\begin{lemma}\label{lem:cosetHzk}
    Let $k \in \ZZ_2^{\times}$. Let $\alpha \in \mathscr{G}$ be an element in the coset $Gz_k$.
    \begin{enumerate}
        \item If $\alpha \restr{T_1}=\id$, then the first descendants of $\alpha$ are in the coset $Gz_k$.
        \item If $\alpha \restr{T_1}\neq \id$, then the first descendants of $\alpha$ are in the coset $Gz_{k^2}$.
    \end{enumerate}
\end{lemma}
\begin{proof} 

Since $z_k\restr{T_1}=\id$, we have $\alpha =(\beta_0,\beta_1)(z_k,\gamma^\ell z_k)$  or $\alpha =(\beta_0,\beta_1)\sigma(z_k,\gamma^\ell z_k)$ for some $\beta_0,\beta_1 \in G$ as $G$ is self-similar. We find that either
\[ \alpha=(\beta_0 z_k, \beta_1 \gamma^\ell z_k) \] or
\[ \alpha=(\beta_0\gamma^{\ell}z_k,\beta_1 z_k)\sigma. \] In the first case, the first descendants are in $Gz_k$, in the second case, the first descendant of $\alpha$ is 
$\beta_0\gamma^{\ell}z_k\beta_1 z_k$ which is in the coset $Gz_kGz_k=Gz_k^2=Gz_{k^2}$.
\end{proof}

\begin{lemma}\label{111estimate}
    Let $\alpha \in \mathscr{G}$ be a stable block such that $\alpha \in Gz_k$ for some $k \neq \pm 1$. Let $\nu$ be the $2$-adic valuation of $\frac{k^2-1}{4}$. Assume that $D_n(\alpha)=\{\beta_n\}$ for all $n\geq 1$. If there is some $t \geq 0$ such that $d(\alpha)_i=1$ for all $i\leq t$, then for any $n \geq \nu +1 + t$, we have 
    $$ s_n(\alpha) \geq s_{n-t}(\beta_t) + 2^{n}-2^{n-t}-t2^{\nu+1}. $$
\end{lemma}
\begin{proof} 
Since $d(\alpha)_{i}=1$ for all $i\leq t$, $i$th descendants of $\alpha$ are all contained in the coset $Gz_k$ for all $1 \leq i\leq t$ by Lemma~\ref{lem:cosetHzk}. 

Let $\nu=v_2(\frac{k^2-1}{4})$. If a descendant $\beta$ of $\alpha$ is in $\mathscr{U}(\gamma)$, then by Corollary~\ref{zksummary}, the number of vertices at level $n$ that are not in a stable cycle of $\beta$ is at most $2^{\nu+1}$ for any $n\geq \nu+1$.

Let $n\geq \nu +1 + t$. Using Proposition~\ref{prop:settledsections}, we calculate
\begin{equation}
    \begin{split}
        s_n(\alpha) &\geq s_{n-1}(\beta_1)+(2^{n-1}-2^{\nu+1}) \\
                    &\geq s_{n-2}(\beta_2)+(2^{n-2}-2^{\nu +1})+(2^{n-1}-2^{\nu +1} )\\
                    \vdots \\
                    &\geq s_{n-t}(\beta_t)+(2^{n-t}-2^{\nu+1})+(2^{n- t+1}-2^{\nu+1})+\ldots +(2^{n-1}-2^{\nu+1}) 
   \end{split}
\end{equation}

Since $n \geq \nu +1 +t$, each summand in the sum is nonnegative. Hence $s_n(\alpha) \geq s_{n-t}(\beta_t) + 2^{n}-2^{n-t}-t2^{\nu+1} $.
\end{proof}

\begin{proposition}\label{lem:111}
    Let $\alpha \in \mathscr{G}$ be a stable block. Assume $|D_n(\alpha)|=1$ for all $n\geq 1$. If there is some $m\geq 0$ such that $d(\alpha)_n=1$ for all $n\geq m$, then $\alpha$ is settled. 
\end{proposition}
\begin{proof}
Let $D_n(\alpha)=\{\beta_n\}$ for all $n\geq 1$. Assume that there is some $m\geq 0$ such that $d(\alpha)_n=1$ for all $n\geq m$.
By Lemma~\ref{lem:Dnsize1}, $\alpha$ is settled if and only if $\beta_m$ is settled. Hence it is enough to show that $\beta_m$ is settled which allows us to assume that $m=0$.

Assume $\alpha$ is in the coset $Gz_k$. Since $\alpha$ is a stable block, we may assume that $k \neq \pm 1$. Let $\nu:= v_2(\frac{k^2-1}{4})$ and let $n\geq \nu +1$. Then applying Lemma~\ref{111estimate} for $t=n-(\nu +1)$, we have 
\begin{equation} \begin{split}
  s_n(\alpha) &\geq s_{\nu +1}(\beta_{n-{\nu+1}})+ 2^{n}-2^{\nu +1}-(n-(\nu +1))2^{\nu+1}\\
              &\geq 2^{n}-n2^{\nu+1} + \nu 2^{\nu +1}.
\end{split}
\end{equation}

 Since $\nu \geq 1$, we have $$s_n(\alpha) \geq 2^{n}-n2^{\nu+1}.$$ Hence, for all $n\geq \nu +1$,
$$    \frac{s_n(\alpha)}{2^n} \geq \frac{2^{n}-n2^{\nu+1}}{2^n}.$$

As $n\to \infty$, the right hand side of the inequality goes to $1$, which means $\alpha$ is settled.
\end{proof}

\begin{lemma}\label{lem:beginwith1}
  Let $\alpha \in \Omega$. Assume $|D_n(\alpha)|=1$ for all $n\geq 1$. If $d(\alpha)_0=d(\alpha)_1=\ldots=d(\alpha)_t=-1$ for some $t\geq 0$, then
  $\alpha$ is settled if and only if $\beta_t$ is settled, where $\beta_t$ is the unique element in $D_t(\alpha)$.
\end{lemma}
\begin{proof}
If $d(\alpha)_0=d(\alpha)_1=\ldots=d(\alpha)_t=-1$ for some $t\geq 0$, then $\Desc_m(\alpha)$ has only one element; namely $\beta_m$ for each $0 \leq m\leq t$. The statement follows from Theorem~\ref{thm:descsettled}. 
\end{proof}

\begin{lemma}\label{lem:estimate}
    Let $\alpha \in \mathscr{G}$. Assume $|D_n(\alpha)|=1$ for all $n\geq 1$. Assume $d(\alpha)_n$ takes both $1$ and $-1$ values infinitely often.
    Then $\alpha$ is settled.
\end{lemma}
\begin{proof}

By Lemma~\ref{lem:beginwith1}, we may assume that $d(\alpha)_0=1$. We first define sequences of positive integer tuples $(r_t,m_t)_{t\geq 1}$ by letting $r_t$ be the length of the $t$th consecutive block of $1$'s in the sequence $(d(\alpha)_n)_{n\geq 0}$ and $m_t$ be the length of the $t$'th consecutive block of $-1$'s immediately following it. Let us call the unique element in $D_n(\alpha)$ by $\beta_n$ for all $n\geq 1$.

Assume $\alpha \in Gz_k$ and let $\nu=\frac{k^2-1}{4}$. For $t\geq 1$, set $n_t:=\nu +1 + 2(m_1+\ldots m_t) + (r_1+\ldots r_t)$. For $n \geq n_t$, Lemma~\ref{111estimate} gives the following estimate:
 \begin{equation}\label{eq:estimate}
         s_n(\alpha) \geq s_{n-r_1}(\beta_{r_1})+(2^n-2^{n-r_1}-r_12^{\nu +1}).
 \end{equation}

 Using Proposition~\ref{prop:settledsections}, we find
 
 \begin{equation}\label{eq:-1-1-1}
         s_{n-r_1}(\beta_{r_1}) \geq 2^{m_1}(s_{n-r_1-m_1}(\beta_{r_1+m_1}))
 \end{equation}

 Combining equations ~\ref{eq:estimate} and ~\ref{eq:-1-1-1}, we have
 \begin{equation} s_n(\alpha)
 \geq 2^{m_1}(s_{n-r_1-m_1}(\beta_{r_1+m_1}))+ (2^{n}-2^{n-r_1}-r_1 2^{\nu +1}).
 \end{equation}
By Lemma~\ref{lem:cosetHzk}, $\beta_{r_1}$ is in the coset $Gz_{k}$ and $\beta_{r_1+m_1}$ is in $Gz_{k^{2^{m_1}}}$. The valuation of $(k^{2^{m_1+1}}-1)/4$ is equal to $\nu+m_1$.

 Applying Lemma~\ref{111estimate} to $s_{n-r_1-m_1}(\beta_{r_1+m_1})$ for $t=r_2$, we have
 \begin{equation}
s_{n-r_1-m_1}(\beta_{r_1+m_1}) \geq s_{n-r_1-r_2-m_1}(\beta_{r_1+r_2+m_1}) + (2^{n-r_1-m_1}-2^{n-r_1-m_1-r_2}-r_2 2^{\nu+m_1+1}))          
 \end{equation}
   

Finally, combing these inequalities, we have 
\begin{equation}
    s_n(\alpha) \geq (2^{n-r_1}-2^{n-r_1-r_2}-r_2 2^{\nu+2 m_1+1}) +(2^{n}-2^{n-r_1}-r_1 2^{1+ \nu})=2^{n}-2^{n-r_1-r_2}-r_2 2^{\nu+2 m_1+1} -r_1 2^{1+ \nu}.
\end{equation}
Repeating this process $t$ times,
we find the following estimate

\begin{equation}
       s_{n}(\alpha)\geq 2^n-2^{n-\sum_{i=1}^{t}r_i}-(r_1 2^{\nu +1}+ \sum_{i=2}^t{r_i 2^{1+\nu +2\sum_{j=1}^{i-1}{m_j}}}).
\end{equation}
for any $n\geq n_t$.

Let $m(t)$ be the sum $m_1+\ldots m_t$ and $r(t)$ be $r_1+\ldots r_t$. Since $\sum_{j=1}^{i-1}{m_j} \leq m(t)$ for all $i=2,\ldots,t-1$, we have

 \begin{equation}
     s_n(\alpha)/2^n \geq 1-(\frac{2^{n-r(t)}+ r_1 2^{\nu +1} + \sum_{i=2}^{t}{r_i 2^{1+ \nu+2m(t)}}}{2^n}).
 \end{equation}
Now, 
\begin{equation}
    \frac{2^{n-r(t)}+\sum_{i=1}^t{r_i 2^{1+ \nu+2m(t)}}}{2^n}=\frac{1}{2^{r(t)}} + \frac{\sum_{i=1}^t{r_i 2^{1+ \nu+2m(t)}}}{2^n}
\end{equation}
Since $d(\alpha)_n$ takes the value $-1$ infinitely many times, $r(t) \to \infty$ as $t$ goes to infinity, it is enough to show the second part of the sum goes to zero as $t \to \infty$. Remember that $n\geq \nu + 2m(t)+r(t) +1$.
\begin{equation}
\begin{split}
    \frac{\sum_{i=1}^t{r_i 2^{1+ \nu+2m(t)}}}{2^n} &\leq \frac{\sum_{i=1}^t{r_i}}{2^{n-1-\nu-2m(t)}} \\
              &\leq  \frac{r(t)}{2^{r(t)}} \to 0 \text{ as } t \to \infty.
\end{split}
\end{equation}

\end{proof}

\begin{proof}[Proof of Theorem~\ref{thm:stablesettled}]
   Let $\alpha$ be a stable block in $\mathscr{G}$. Using Lemma~\ref{lem:Dempty}, we may assume that $|D_n(\alpha)|=1$ for all $n\geq 1$. Using  Lemma ~\ref{lem:-1-1-1} and Proposition~\ref{lem:111} we may assume $d(\alpha)_n$ takes the values $\pm 1$ both infinitely many times. Lemma~\ref{lem:estimate} shows in this case that $\alpha$ is settled.
   Hence, we showed that a stable block is settled.
\end{proof}

Using Theorem~\ref{thm:stablesettled}, one can demonstrate more explicit examples of settled elements in $\mathscr{G}$.

\begin{proposition}\label{prop:azk}
 Let $m\in \ZZ_2$ and $k \neq \pm 1 \in \ZZ_2^{\times}$. The element $a_i\gamma^mz_k$ is settled for any $i=1,\ldots,r$. In particular, $a_iz_k$ is settled for all $i=1,\ldots,r$.  
\end{proposition}
\begin{proof}
    We calculate for $m\in \ZZ_2$ and $k \in \ZZ_2^{\times}$ that,
    \begin{equation} \label{i=1}
    \Desc_1(a_1\gamma^mz_k) =\begin{cases}
      \{a_r\gamma^{\frac{m-1}{2}}z_k,  \gamma^{\frac{m+1}{2}+\ell}z_k \} & \text{ if } v_2(m)=0 \\
      \{a_r\gamma^{\frac{(1+k)m}{2}+\ell}z_{k^2}\} & \text{ otherwise }
    \end{cases}
    \end{equation}
   and for $1 <i\leq r$,
   \begin{equation} \label{i>1}
    \Desc_1(a_i\gamma^mz_k) =\begin{cases}
      \{ a_{i-1}\gamma^{\frac{m}{2}}z_k,  \gamma^{\frac{m}{2}+\ell}z_k\} & \text{ if } v_2(m)>0 \\
      
      \{ a_{i-1}\gamma^{\frac{(1+r)}{2}+k\frac{(r-1)
      
      }{2}+\ell}z_{k^2}\} & \text{ otherwise}
    \end{cases}
    \end{equation}
Hence the element $a_i\gamma^mz_k$ is a block for all $1\leq i\leq r$. The equation \ref{i=1} and \ref{i>1} tells us that a first descendant of $a_i\gamma^mz_k$ is either in $\mathscr{U}(\gamma)$ or it is $a_{i-1}\gamma^rz_{k'}$ for some $k'\neq \pm 1$ and $r\in \ZZ_2$. Since the elements of $\mathscr{U}(\gamma)$ are stable blocks, we only need to consider the element $a_{i-1}\gamma^rz_{k'}$. We already discussed that this element is a block. Since an $n$th descendant of an element $\alpha$ is a first descendant of an $(n-1)$st descendant of $\alpha$, it follows that $a_i\gamma^mz_k$ is a stable block for all $1\leq i \leq r$, $m\in \ZZ_2$ and $k\neq \pm 1 \in \ZZ_2^{\times}$. Hence, it is settled by Theorem~\ref{thm:stablesettled}.  
\end{proof}

\section{$\mathscr{G}$ is densely settled}

In this final section, we prove that $\mathscr{G}$ is densely settled.
Our strategy is to identify an explicit dense subset of $\mathscr{G}$ and then
show that every element of this subset is settled.
As discussed in Remark~\ref{rem:Gar}, this immediately implies that the group $\ga$ is also densely settled when $K$ is finitely generated over its prime field.

\begin{proposition}\label{prop:denseset}
The set
\begin{equation}\label{denseset}
\mathcal{A}
:=
\left\{
\alpha z_k \;\middle|\;
\alpha \text{ is a finite word in } \{a_1,\ldots,a_r\},\;
k>1 \text{ is an odd integer}
\right\}
\end{equation}
is dense in $\mathscr{G}$.
\end{proposition}

\begin{proof}
Let $\beta\in \mathscr{G}$. We will show that for any $n\geq 1$, there is an element of $\mathcal{A}$ that agrees with $\beta$ when restricted to $T_n$.

By definition of $\mathscr{G}$, $\beta$ can be written in the form $\beta=g z_k$ for some
$g\in G$ and $k\in\mathbb Z_2^\times$. 

Assume that $n > v_2(k-1)$. Writing $k=\sum_{i=0}^{\infty}c_i2^i$ for some $0\leq c_i \leq 1$, we let $k_n:=\sum_{i=0}^{n-1}c_i2^i$. Notice that $v_2(k-k_n) \geq n$.
Since $n > v_2(k-1)$, we have that $c_{i}=1$ for some $1\leq i \leq n-1$ and hence the integer $k_n$ is greater than $1$. Hence we have
\[
z_k\restr{T_n}=z_{k_n}\restr{T_n}.
\]
by Lemma~\ref{lem:zkcongruence}.
It follows that
\[
\beta\restr{T_n} = g z_{k_n}\restr{T_n},
\]
For $n \leq v_2(k-1)$, one can take $k_n=1+2^n$. We adopt the convention of $v_2(0)=\infty$, so this covers $k=1$ as well.
By Lemma~\ref{kidentitylevel}, we have $z_k \restr{T_n}=z_{k_n}\restr{T_n}=\id$.

Hence, we showed for any $n\geq 1$ that $\beta \restr{T_n}=gz_{k_n}$ where $k_n$ is an odd integer.

Since $G$ is topologically generated by $a_1,\ldots,a_r$ and $G_n$ is finite, for each $n\geq 1$, there is an element $g_n \in G$ such that
 $g\restr{T_n}=g_n \restr{T_n}$ and $g_n$ is a finite word on $a_1,\ldots,a_r$, i.e. $g_n \in \langle a_1,\ldots,a_r\rangle$.
This proves that the set $\mathcal{A}$ is dense in $\mathscr{G}$.
\end{proof}

However, the descendants of an element of $\mathcal{A}$ do not necessarily lie in $\mathcal{A}$.
For example, the first descendant of $a_1^3 z_k$ is $a_2^2 \gamma^{\ell} z_k a_2 z_k$, which lies in the coset $G z_k^2$.
There is no guarantee that this element belongs to $\mathcal{A}$.
For this reason, we introduce a larger set that is stable under taking descendants and is suitable for an induction on a length function defined below.

Fix $k\in\mathbb Z_2^\times$ with $k\neq \pm 1$, and consider the coset $G z_k\subseteq \mathscr{G}$.
To capture all elements appearing in the dense set $\mathcal{A}$ defined in~\eqref{denseset}, we define
\[
\mathfrak{V}_k
:=
\left\{
\beta\in G z_k \;\middle|\;
\beta \text{ is represented by a finite word in the alphabet}
\{a_1,\dots,a_r,\gamma^m,z_s : s \neq \pm 1 \in\mathbb Z_2^\times, m\in \ZZ_2\}
\right\}.
\]
We then set
\[
\mathfrak{V} := \bigcup_{k\neq\pm 1} \mathfrak{V}_k.
\]

The set $\mathfrak{V}$ clearly contains the dense subset $\mathcal{A}$, but it may also contain additional elements.
For instance, an element of the form $z_k a_1$ lies in $\mathfrak{V}$, but need not be expressible as $a_i^m z_k$ for any integer $m$.
This reflects the fact that conjugation of the generators $a_i$ by $z_k$ is not explicitly understood.

Finally, we observe that if $\beta\in \mathfrak{V}$, then every descendant of $\beta$ also lies in $\mathfrak{V}$.

\subsection*{The $a$-length}

Let $\alpha$ be a finite word in the alphabet 
$\{a_1,\dots,a_r,\gamma^m,z_k : k \neq \pm 1 \in \ZZ_2^{\times}, m\in \ZZ_2\}$.
Let $|\alpha|_{a_1,\dots,a_r}$ denote the number of occurrences of letters
$a_1,\dots,a_r$ in $\alpha$.

For $\beta\in \mathfrak{V}$, define the \emph{$a$-length} of $\beta$ by
\[
|\beta|_a := 
\min\{\;|\rho|_{a_1,\dots,a_r} :
\rho \text{ is a finite word representing }\beta \;\}.
\]

\begin{lemma}\label{ai not in N(gamma)}
    The element $a_i$ is not in $N(\gamma)$ for any $i=1,\ldots,r$.
\end{lemma}
\begin{proof}
    By Proposition~\ref{normalizer of gamma}, each element of $N(\gamma)$ is of the form $\gamma^mz_k$ for some $k \in \ZZ_2^{\times}$ and $m \in \ZZ_2$. If $k=-1$, then $\gamma^m z_{-1}$ has finite order for all $m\in \ZZ_2$ since 
    $\gamma^m z_{-1} \gamma^mz_{-1}=\gamma^m\gamma^{-m}=\id$. By \cite[Lemma~2.2.5]{Pinkpolyn} $a_i$ has an infinite order for all $i=1,\ldots,r$, hence if $a_i$ is  in $N(\gamma)$, then it must be of the form $\gamma^m z_k$ where $k\neq -1$.
    
    By Corollary~\ref{zksummary}, if $k\neq \pm 1$, $\gamma^mz_k$ is settled for any $m\in \ZZ_2$. If $k=1$ and $m\neq 0$, then  $\gamma^mz_k=\gamma^m$ which is also settled. Therefore, if $a_i \in N(\gamma)$, then it must be settled. We observe that $a_i$ for $i>1$ is not settled since the number of vertices that are in a stable cycle is at most $1/2$ of the vertices on that level as they act as identity on one of the subtrees at level one. Finally, since $a_1=(a_r,\id)\sigma$, it is settled if and only if $a_r$ is settled by Corollary~\ref{cor:descsettled}. Hence $a_i$ cannot be in $N(\gamma)$ for any $i=1,\ldots,r$.
\end{proof}

Elements of $N(\gamma)$ have $a$-length $0$, and each generator $a_i$ has $a$-length $1$ since $a_i$ is not in $N(\gamma)$ for any $1\leq i\leq r$ by Lemma~\ref{ai not in N(gamma)}. Let $\beta=(\beta_0,\beta_1)\tau$. By the recursion relations~\ref{generators}, each letter $a_i$ in a minimal word for $\beta$
contributes a single letter $a_{i-1}$ to the word representing either $\beta_0$ or $\beta_1$. Therefore, if $\beta=(\beta_0,\beta_1)\tau \in \mathfrak{V}$, then
\[
|\beta_0|_a + |\beta_1|_a \;\le\; |\beta|_a.
\]
Inductively, this argument shows that if $\beta'$ is a descendant of $\beta \in \mathfrak{V}$, then $|\beta'|_a \leq |\beta|_a$.

\begin{proposition}\label{prop:VkSettled}
Every element of $\mathfrak{V}$ is settled.
\end{proposition}
\begin{proof}
Let $\beta\in \mathfrak{V}$, and proceed by induction on $|\beta|_a$.

If $|\beta|_a=0$, then it is in $N(\gamma)$ and since $k\neq \pm 1$, it is uniformly settled. If $|\beta|_a=1$, then either 
$\beta=a_i\gamma^m z_k$ for some $i$,
in which case Proposition~\ref{prop:azk} implies $\beta$ is settled,
or $\beta$ is conjugate to such an element, and hence settled as well. The second part follows from the the definition of $z_k$ and the fact that $\alpha_0\alpha_1=\alpha_0(\alpha_1\alpha_0)\alpha_0^{-1}$ in any group.

Assume inductively that every $\alpha\in \mathfrak{V}$
with $|\alpha|_a < |\beta|_a$ is settled.
If $\beta$ is a stable block, then it is settled by Theorem~\ref{thm:stablesettled}.

Otherwise, either $\beta$ is not a block or $\beta$ has a descendant that is not a block. The proof is similar in both cases.
Assume  for some $n\geq 1$, there exists an $n$th descendant
\[
\alpha=(\alpha_0,\alpha_1)
\]
of $\beta$ such that neither $\alpha_0$ nor $\alpha_1$ lies in $\mathscr{U}(\gamma)$.
Therefore, $|\alpha_0|_a>0$ and $|\alpha_1|_a>0$, and 
the inequality 
$$|\alpha_0|_a + |\alpha_1|_a \le |\alpha|_a \le |\beta|_a $$
implies that both $|\alpha_0|_a$ and $|\alpha_1|_a$
are strictly smaller than $|\beta|_a$.

By the induction hypothesis, both $\alpha_0$ and $\alpha_1$ are settled.
Hence $\alpha$ is settled, and Theorem~\ref{thm:descsettled} then implies that $\beta$ is settled.
\end{proof}

Proposition~\ref{prop:VkSettled} shows that every element of the set $\mathcal{A}$ is settled. Thus, by Proposition~\ref{prop:denseset},
$\mathscr{G}$ is densely settled.

\begin{corollary}\label{cor:final}
    Let $K$ be a field of characteristic not equal to $2$. Assume that $K$ is finitely generated over its prime field. Let $f(x)$ be a quadratic polynomial defined over $K$. Assume that the postcritical orbit of $f$ is periodic. Then the arithmetic iterated monodromy group $\ga$ of $f$ is densely settled.
\end{corollary}

\bibliographystyle{amsplain}
\bibliography{refsettledarboreal}

\end{document}